\documentclass[11pt]{amsart}

\usepackage[mathscr]{eucal}
\usepackage{amsmath,amssymb,amsfonts,amsthm,enumerate}
\usepackage{hyperref}

\newtheorem{theorem}{Theorem}[section]
\newtheorem{lemma}[theorem]{Lemma}
\newtheorem{corollary}[theorem]{Corollary}
\newtheorem{proposition}[theorem]{Proposition}
\newtheorem{conjecture}[theorem]{Conjecture}
\newtheorem{theirtheorem}{Theorem}

\newcommand{\Z}{\mathbb Z}

\DeclareMathOperator{\ord}{ord}

\DeclareMathOperator{\supp}{Supp}

\newcommand{\la}{\langle}
\newcommand{\ra}{\rangle}
\newcommand{\be}{\begin{equation}}
\newcommand{\ee}{\end{equation}}
\newcommand{\und}{\;\mbox{ and }\;}

\newcommand{\ber}{\begin{eqnarray}}
\newcommand{\eer}{\end{eqnarray}}
\newcommand{\Sum}[2]{\underset{#1}{\overset{#2}{\sum}}}
\newcommand{\Summ}[1]{\underset{#1}{\sum}}

\newcommand{\wtilde}{\widetilde}

\newcommand{\Fc}{\mathcal F}
\newcommand{\vp}{\mathsf v}
\newcommand{\h}{\mathsf h}

\DeclareSymbolFont{goo}{OMS}{cmsy}{b}{n}
\DeclareMathSymbol{\gooT}{\mathalpha}{goo}{"1}
\newcommand{\bdot}{\mathbin{\gooT}}
\begin{document}

\title[A Multiplicative Property for Zero-Sums I]{A  Multiplicative Property for Zero-Sums I}
%{A Multiplicative Property for Maximal Length Sequences in $(\Z/n\Z)\times (\Z/n\Z)$ Avoiding Zero-Sums of Moderate Length}
\author{David J. Grynkiewicz}
\address{Department of Mathematical Sciences\\ University of Memphis\\ Memphis, TN 38152\\
USA}
\email{diambri@hotmail.com}
\author{Chao Liu}

\address{Department of Mathematical Sciences\\ University of Memphis\\ Memphis, TN 38152\\
USA}
\email{chaoliuac@gmail.com}


\begin{abstract}
Let $G=(\mathbb Z/n\mathbb Z)\times(\mathbb Z/n\mathbb Z)$ and let $k\in [0,n-1]$. We study the structure of sequences of terms from $G$ with maximal length $|S|=2n-2+k$ that fail to contain a nontrivial zero-sum subsequence of length at most $2n-1-k$. For $k\leq 1$, this is the inverse question for the Davenport Constant. For $k=n-1$, this is the inverse question for the $\eta(G)$ invariant concerning short zero-sum subsequences. The structure in both these cases (known respectively as Property B and Property C) was established in a two-step process: first verifying the multiplicative property that, if the structural description holds when $n=n_1$ and $n=n_2$, then it holds when $n=n_1n_2$, and then resolving the case $n$ prime separately. When $n$ is prime, the structural characterization for $k\in [2,\frac{2n+1}{3}]$ was recently established, showing $S$ must have the form $S=e_1^{[n-1]}\boldsymbol{\cdot}e_2^{[n-1]}\boldsymbol{\cdot} (e_1+e_2)^{[k]}$ for some basis $(e_1,e_2)$ for $G$. It was conjectured that this also holds for $k\in [2,n-2]$ (when $n$ is prime). In this paper, we extend this conjecture  by dropping  the restriction that $n$ be prime and  establish the following multiplicative result. Suppose $k=k_mn+k_n$ with $k_m\in [0,m-1]$ and $k_n\in [0,n-1]$. If the conjectured structure holds for $k_m$ in $(\mathbb Z/m\mathbb Z)\times (\mathbb Z/m\mathbb Z)$ and for $k_n$ in $(\mathbb Z/n\Z)\times (\mathbb Z/n\mathbb Z)$, then it holds for $k$ in $(\mathbb Z/mn\mathbb Z)\times (\mathbb Z/mn\mathbb Z)$. This reduces the full characterization question for $n$ and $k$ to the prime case. Combined with known results, this unconditionally establishes the structure for extremal sequences in $G=(\mathbb Z/n\mathbb Z)\times (\mathbb Z/n\mathbb Z)$ in many cases, including when  $n$ is only divisible by primes at most $7$, when $n\geq 2$ is a prime power and $k\leq \frac{2n+1}{3}$, or when  $n$ is composite and $k=n-d-1$ or $n-2d+1$ for a proper, nontrivial divisor $d\mid n$.
\end{abstract}

\maketitle

\section{Introduction and Preliminaries}

Regarding combinatorial notation for sequences and subsums, we utilize the standardized system surrounding multiplicative strings as outlined in the references \cite{Ger-Ruzsa-book} \cite{Alfredbook} \cite{Gbook}. For the reader new to this notational system, we begin with a self-contained review.

\subsection*{Notation} All intervals will be discrete, so for $x,\,y\in \Z$, we have $[x,y]=\{z\in \Z:\; x\leq z\leq y\}$. More generally, if $G$ is an abelian group, $g\in G$, and $x,\,y\in \Z$, then
$$[x,y]_g=\{xg,(x+1)g,\ldots, yg\}.$$ For $G=C_n\oplus C_n$
a (ordered) \textbf{basis} for $G$ is a pair $(e_1,e_2)$ of elements $e_1,e_2\in G$ such that $G=\la e_1\ra\oplus \la e_2\ra=C_n\oplus C_{n}$. For subsets $A_1,\ldots,A_k\subseteq G$, their sumset is defined as $A_1+\ldots+A_k=\{a_1+\ldots+a_k:\; a_i\in A_i\mbox{ for $i\in [1,k]$}\}$.

Let $G$ be an abelian group. In the tradition of Combinatorial Number Theory, a sequence of terms from $G$ is a finite, unordered string of elements from $G$. We let $\Fc(G)$ denote the free abelian monoid with basis $G$, which consists of all (finite and unordered) sequences $S$ of terms from $G$ written as multiplicative strings using the boldsymbol $\bdot$ . This means a sequence $S\in \Fc(G)$ has the form  $$S=g_1\bdot\ldots\bdot g_\ell$$ with $g_1,\ldots,g_\ell\in G$ the terms in $S$.
Then $$\vp_g(S)=|\{i\in [1,\ell]:\; g_i=g\}|$$ denotes the multiplicity of the terms $g$ in $S$, allowing us to represent a sequence $S$ as $$S={\prod}^\bullet_{g\in G}g^{[\vp_g(S)]},$$ where $g^{[n]}={\underbrace{g\bdot\ldots\bdot g}}_n$ denotes a sequence consisting of the term $g\in G$ repeated $n\geq 0$ times.
 The maximum multiplicity of a term of $S$ is the height of the sequence, denoted $$\h(S)=\max\{\vp_g(S):\; g\in G\}.$$ The support of the sequence $S$ is the subset of all elements of $G$ that are contained in $S$, that is, that occur with positive multiplicity in $S$, which is denoted $$\supp(S)=\{g\in G:\; \vp_g(S)>0\}.$$
 The length of the sequence $S$ is $$|S|=\ell=\Summ{g\in G}\vp_g(S).$$  A sequence $T\in \Fc(G)$ with $\vp_g(T)\leq \vp_g(S)$ for all $g\in G$ is called a subsequence of $S$, denoted $T\mid S$, and in such case, $S\bdot T^{[-1]}=T^{[-1]}\bdot S$ denotes the subsequence of $S$ obtained by removing the terms of $T$ from $S$, so $\vp_g(S\bdot T^{[-1]})=\vp_g(S)-\vp_g(T)$ for all $g\in G$.

 Since the terms of $S$ lie in an abelian group, we have the following notation regarding subsums of terms from $S$. We let $$\sigma(S)=g_1+\ldots+g_\ell=\Summ{g\in G}\vp_g(S)g$$ denote the sum of the terms of $S$ and call $S$ a \textbf{zero-sum} sequence when $\sigma(S)=0$. A \textbf{minimal zero-sum} sequence is a zero-sum sequence that cannot have its terms partitioned into two proper, nontrivial zero-sum subsequences.  For $n\geq 0$, let
 \begin{align*}&\Sigma_n(S)=\{\sigma(T):\; T\mid S, \; |T|=n\},\quad
 \Sigma_{\leq n}(S)=\{\sigma(T):\; T\mid S, \; 1\leq |T|\leq n\},\quad\und\quad\\
 &\Sigma(S)=\{\sigma(T):\; T\mid S, \; |T|\geq 1\}
\end{align*}
 denote the variously restricted collections of  subsums of $S$.
 The sequence $S$ is \textbf{zero-sum free} if $0\notin \Sigma(S)$.
 Finally, if $\varphi:G\rightarrow G'$ is a map, then $$\varphi(S)=\varphi(g_1)\bdot\ldots\bdot \varphi(g_\ell)\in \Fc(G')$$ denotes the sequence of terms from $G'$ obtained by applying $\varphi$ to each term from $S$.

\subsection*{Background} Let $G$ be a finite abelian group. A classic topic in Combinatorial Number Theory is the study of conditions on sequences that ensure the existence of zero-sum subsequences with prescribed properties. Apart from the intrinsic combinatorial interest in such questions, they are also important when studying properties of factorization in Krull Domains and, more generally, in (Transfer) Krull Monoids. See \cite{Alfredbook} \cite{Ger-Ruzsa-book}.

The most classic zero-sum invariant is the Davenport Constant $\mathsf D(G)$, defined as the minimal length  such that any sequence of terms from $G$ with length at least $\mathsf D(G)$ contains a nontrivial zero-sum subsequence. It is well-known  that $\mathsf D(G)$ can be equivalently defined as  the maximal length of a minimal zero-sum sequence. Indeed, $\mathsf D(G)-1$ is, by definition, the maximal length of a zero-sum free sequence $T$, and then one readily notes that $T\bdot -\sigma(T)$ will be a minimal zero-sum sequence of length $\mathsf D(G)$. This shows there are minimal zero-sums of length $\mathsf D(G)$. Conversely, if $S$ is any minimal zero-sum, then $S\bdot g^{[-1]}$ is zero-sum free for any $g\in \supp(G)$, ensuring no minimal zero-sum can have length exceeding $\mathsf D(G)$.

The precise value of $\mathsf D(G)$ is open in general and known only for a few small families of abelian groups, including $p$-groups and  groups of rank at most two \cite{Alfredbook}. In particular \cite[Theorem 5.8.3]{Alfredbook}, $$\mathsf D(C_n\oplus C_{n})=2n-1$$ for $n\geq 1$. This is an old result of Olson  \cite{olson-rk2} or van Emde Boas and Kruyswijk \cite{Boas-K-rk2-Dav} whose proof required a more refined constant $\eta(G)$, defined as the  minimal length  such that any sequence of terms from $G$ with length at least $\eta(G)$ contains a nontrivial zero-sum subsequence of length at most $\exp(G)$. For $G=C_n\oplus C_n$, we have \cite{olson-rk2} \cite{Boas-K-rk2-Dav} \cite[Theorem 5.8.3]{Alfredbook} $$\eta(C_n\oplus C_{n})=3n-2.$$

As a special case of a more general constant \cite{Cohen} \cite{Cong-zs}, Delorme, Ordaz and Quiroz introduced \cite{S_<k-invariant-origin-delorme-ordaz} the refined constant $\mathsf s_{\leq \ell}(G)$ defined as the minimal length such that any sequence of terms from $G$ with length at least $\mathsf s_{\leq \ell}(G)$ contains a nontrivial zero-sum subsequence of length at most $\ell$, i.e., $$|S|\geq \mathsf s_{\leq \ell}(G)\quad \mbox{ implies }\quad 0\in \Sigma_{\leq \ell}(S).$$
Relations between $\mathsf s_{\leq \ell}(G)$ and Coding Theory may be found in \cite{Cohen}, and other related works dealing with $\mathsf s_{\leq \ell}(G)$ include \cite{freeze}
\cite{Roy-s<k} \cite{Gao-liu}.
When $\ell<\exp(G)$, we have  $\mathsf s_{\leq \ell}(G)=\infty$; when $\ell=\exp(G)$, we have  $\mathsf s_{\leq \ell}(G)=\eta(G)$; and when $\ell\geq \mathsf D(G)$, we have $\mathsf s_{\leq \ell}(G)=\mathsf D(G)$. Thus, concerning the constant $\mathsf s_{\leq \ell}(G)$, the range of interest is $\ell\in [\exp(G),\mathsf D(G)]$, and $\mathsf s_{\leq \ell}(G)$ interpolates between the well-studied invariants  $\eta(G)$ and $\mathsf D(G)$.
For the case of $G=C_n\oplus C_n$, Chulin Wang and Kevin Zhao determined the exact value of $s_{\leq \ell}(G)$,
showing \cite{kevin-S_<k-invariant}
\begin{align*}\mathsf s_{\leq D-k}(C_{n}\oplus C_{n})=D+k,\quad  \mbox{ for $k\in [0,D-\exp(G)]$,}\end{align*} where $D=\mathsf D(C_{n}\oplus C_{n})$. Since $\mathsf D(C_{n}\oplus C_{n})=2n-1$, this can be restated as \begin{align*}\mathsf s_{\leq 2n-1-k}(C_{n}\oplus C_{n})=2n-1+k,\quad  \mbox{ for $k\in [0,n-1]$.}\end{align*}

With the value $\mathsf s_{\leq 2n-1-k}(C_n\oplus C_{n})=2n-1+k$ established, there arises the associated inverse question characterizing all extremal sequences having  maximal length
$2n-2+k=\mathsf s_{\leq 2n-1-k}(C_n\oplus C_{n}) -1$ with $0\notin \Sigma_{\leq 2n-1-k}(S)$. For $k=0$, this amounts to  characterizing all zero-sum free sequences of maximal length $2n-2=\mathsf D(G)-1$.
For $k=1$, this amounts to characterizing all minimal zero-sum sequences of maximal length $2n-1=\mathsf D(G)$. In view of our previous commentary,  these two cases are  equivalent to each other, the extremal sequences for $k=0$ being simply the extremal sequences for $k=1$ with any one term removed.  For $k=n-1$, this amounts to characterizing all extremal sequences of length  $3n-3=\eta(G)-1$ with $0\notin \Sigma_{\leq n}(S)$.

The precise structure in both the case $k\leq 1$ and the case $k=n-1$ is known. For $k\leq 1$, this is achieved by combining the individual results of Gao, Geroldinger, Grynkiewicz and Reiher  from \cite{propBGaoGer-multof2} \cite{PropB}  \cite{PropBfix} \cite{Reiher-propB}  with the numerical verification of the case when $n=9$ \cite{PropB-smallcase}. The characterization of the extremal sequences for the Davenport constant has since proved quite useful, for instance being employed as machinery for the results in  \cite{propB-app4}  \cite{PropB-app6} \cite{PropB-app7}  \cite{PropB-app2} \cite{propB-app1} \cite{schmid-propC-exact} \cite{PropB-app8}.
Since we will need to use both known  cases heavily, we introduce  some terminology.

A sequence $S$ of terms from $G=C_n\oplus C_{n}$ is said to have \textbf{Property A} if there is a basis $(e_1,e_2)$ for $G=C_n\oplus C_{n}$ such that $\supp(S)\subseteq \{e_1\}\cup \big(\la e_1\ra+e_2\big)$. We say that the \emph{group} $G=C_n\oplus C_{n}$ has \textbf{Property A} if every minimal zero-sum sequence $S$ with $|S|=\mathsf D(G)=2n-1$ satisfies Property A.
A sequence $S$ of terms from $G=C_n\oplus C_{n}$ is said to have \textbf{Property B} if $\mathsf h(S)=\exp(G)-1=n-1$, that is, $S$ has some term $e_1$ with multiplicity $n-1$. We say that the \emph{group} $G=C_n\oplus C_{n}$ has \textbf{Property B} if every minimal zero-sum sequence $S$ with $|S|=\mathsf D(G)=2n-1$ satisfies Property B.
A simple argument shows that a minimal zero-sum sequence $S$ with $|S|=\mathsf D(G)=2n-1$ satisfying  Property A with basis $(e_1,e_2)$ has the form  \be\label{mzf-form}S=e_1^{n-1}\bdot {\prod}^\bullet_{i\in [1,n]}(x_ie_1+e_2)\ee for some $x_1,\ldots,x_n\in [0,n-1]$ with $x_1+\ldots+x_n\equiv 1\mod n$. In particular, $S$ satisfies Property B.  It is also not hard to show (see \cite{PropBfix}) that a minimal zero-sum sequence $S$ with $|S|=\mathsf D(G)=2n-1$ that satisfies Property B, say with $\vp_{e_1}(S)=\mathsf h(S)=n-1$, has a basis  $(e_1,e_2)$ such that $S$ has the form given in \eqref{mzf-form}, and thus satisfies Property A with respect to the basis $(e_1,e_2)$. Note, when $S$ has two distinct elements $e_1$ and $e_2$ both with multiplicity $n-1$, this ensures $S=e_1^{[n-1]}\bdot e_2^{[n-1]}\bdot (e_1+e_2)$ with $(e_1,e_2)$ and $(e_2,e_1)$ both bases for $G$ with respect to which $S$ satisfies Property A.

Any sequence $S$ having the form given in \eqref{mzf-form} is easily seen to be a minimal zero-sum sequence. The converse, that \emph{every} minimal zero-sum sequence $S$ of maximal length $|S|=\mathsf D(G)=2n-1$ must have the form given in \eqref{mzf-form}, is the  structural characterization of extremal sequences for the Davenport constant that was previously alluded to, which required  several years and the combined effort of all the results from \cite{propBGaoGer-multof2} \cite{PropB} \cite{PropBfix} \cite{Reiher-propB} (as well as the individual verification of the case $n=9$ \cite{PropB-smallcase}).

The precise structure of all extremal sequences for $k=n-1$ was achieved in \cite{schmid-propC-exact} \cite{PropB->PropC}, and relies on the characterization in the case $k\leq 1$. We continue with the commonly used terminology in this case.

A sequence $S$ of terms from $G=C_n\oplus C_{n}$ is said to have \textbf{Property C} if every term of $S$ has multiplicity $n-1$.
We say that the \emph{group} $G=C_n\oplus C_{n}$ has \textbf{Property C} if every sequence $S$ with $|S|=\eta(G)-1=3n-3$ and $0\notin \Sigma_{\leq n}(S)$  must satisfy Property C. It was shown in \cite{PropB->PropC} that, assuming  Property B (equivalently, property A) holds for $G$, then every sequence $S$ with $|S|=\eta(G)-1=3n-3$ and $0\notin \Sigma_{\leq n}(S)$  must satisfy Property C, i.e., that Property A/B holding for $G=C_n\oplus C_n$ implies that Property C holds for $G$. Rather surprisingly, in contrast to the case for Property A/B, this does not easily yield a precise structural description of all possibilities for extremal sequences $S$ when $k=n-1$. For $n=p$ prime, a derivation of the precise characterization from Property C can be found in \cite{emde-propC-prime}, and the derivation of the precise characterization from Property C in the general case (when $n$ may be composite) follows from a  result of  Schmid \cite{schmid-propC-exact}.
All such sequences satisfy Property A, and thus have the form \be\label{propC-form}S=e_1^{[n-1]}\bdot e_2^{[n-1]}\bdot (xe_1+e_2)^{[n-1]}\ee for some basis $(e_1,e_2)$ for $G=C_n\oplus C_n$ and some $x\in [1,n-1]$ with $\gcd(x,n)=1$.

In view of the discussion above, the inverse problem for $\mathsf s_{\leq 2n-1-k}(C_n\oplus C_{n})$ is complete for the boundary values $k\leq 1$ and $k=n-1$. For the interior values $k\in [2,n-2]$ (and thus, for $n\geq 4$), a precise characterization of all extremal sequences $S$ with length $|S|=\mathsf s_{\leq 2n-1-k}(C_n\oplus C_{n})-1$ but $0\notin \Sigma_{\leq 2n-1-k}(S)$ is still open. There is partial progress in the case  when $n=p$ is prime achieved in \cite{S_<k-invariant2/3p}, where the precise structure is characterized for $G=C_p\oplus C_p$ when $k\in [2,\frac{2p+1}{3}]$ with $p\geq 5$, showing all such extremal sequences must have the form $$S=e_1^{[p-1]}\bdot e_2^{[p-1]}\bdot (e_1+e_2)^{[k]}$$ for some basis $(e_1,e_2)$ for $G=C_p\oplus C_p$.
It was conjectured in \cite{S_<k-invariant2/3p} \cite{kevin-S_<k-invariant} that the same structure should hold for any $k\in [2,p-2]$. Naturally extending this conjecture to composite values, we obtain the following conjecture that, if true,  would fully characterize the structure of all extremal sequences for the zero-sum invariant $\mathsf s_{\leq 2n-1-k}(C_n\oplus C_n)$.

\begin{conjecture}
\label{conj-shortzs} Let $n\geq 2$,  let $G=C_n\oplus C_{n}$,  let $k\in [0,n-1]$, and let $S$ be a sequence of terms from $G$ with  $$|S|=2n-2+k\quad\und\quad 0\notin \Sigma_{\leq 2n-1-k}(S).$$ Then there exists  a basis  $(e_1,e_2)$ for $G$  such that the following hold.
\begin{itemize}
\item[1.]  If $k=0$, then $S\bdot g$ satisfies the description given in Item 2,  where $g=-\sigma(S)$.
\item[2.]  If $k=1$, then  $$S=e_1^{[n-1]}\bdot {\prod}_{i\in [1,mn]}^\bullet (x_ie_1+e_2),$$ for some $x_1,\ldots,x_{mn}\in [0,n-1]$ with $x_1+\ldots+x_{mn}\equiv 1\mod n$.

\item[3.] If $k\in [2,n-2]$, then $$S=e_1^{[n-1]}\bdot e_2^{[n-1]}\bdot(e_1+e_2)^{[k]}.$$
\item[4.] If $k=n-1$, then  $$S=e_1^{[n-1]}\bdot e_2^{[n-1]}\bdot(xe_1+e_2)^{[n-1]},$$ for some  $x\in [1,n-1]$ with $\gcd(x,n)=1$.
\end{itemize}
\end{conjecture}

Per the discussion above, Parts 1, 2 and 4 in Conjecture \ref{conj-shortzs} are known, and Part 3 holds  when $n=p$ is prime and $k\leq \frac{2p+1}{3}$. It was also shown in \cite{S_<k-invariant2/3p} that Conjecture \ref{conj-shortzs}.3 holds when  $n=p^s\geq 5$ is a prime power with $k\leq \frac{2n+1}{3}$ and  $p\nmid k$.   In general, we say that Conjecture \ref{conj-shortzs} holds for $k$ in $C_n\oplus C_{n}$ if Conjecture \ref{conj-shortzs} is true when $G=C_n\oplus C_{n}$ for the given value $k\in [0,n-1]$.
The main goal of this paper is Theorems \ref{thm-mult},  which shows that the structural description given in Conjecture  \ref{conj-shortzs}.3 is multiplicative, thus reducing the full characterization problem  for $\mathsf s_{\leq 2n-1-k}(C_n\oplus C_{n})$ to the case when   $n=p$ prime, so the case when $G=C_p\oplus C_p$ with $p\geq 11$ prime (in view of Corollary \ref{cor-mult-2,3}). This reduction to the prime case is the main aim of the paper and emulates the strategy successfully used to characterize the extremal sequences for the Davenport Constant (the case $k\leq 1$), where the characterization problem was first reduced by a similar multiplicative result to the prime case \cite{propBGaoGer-multof2} \cite{PropB} \cite{PropBfix}, with the prime case later resolved by independent methods \cite{Reiher-propB}. We remark that Schmid later reduced the characterization of extremal sequences for the Davenport Constant, in  a general rank two abelian group, to the case $C_n\oplus C_n$ \cite{Schmid-propB}, and a forthcoming work \cite{chaoII} aims to similarly extend our methods to general rank two abelian groups.

\begin{theorem}\label{thm-mult}
Let $n,m\geq 2$ and let $k\in [0,mn-1]$ with   $k=k_mn+k_n$, where  $k_m\in [0,m-1]$ and $k_n\in [0,n-1]$. Suppose Conjecture \ref{conj-shortzs} holds for $k_n$ in $C_{n}\oplus C_{n}$ and either Conjecture \ref{conj-shortzs} also holds for $k_m$ in $C_{m}\oplus C_{m}$ or else $k_n\geq 1$, $k_m\in [1,m-2]$ and Conjecture \ref{conj-shortzs} also  holds for $k_m+1$ in $C_m\oplus C_m$. Then Conjecture \ref{conj-shortzs} holds for $k$ in $C_{mn}\oplus C_{mn}$.
\end{theorem}

While the reduction to the prime case is our main motivating goal, nonetheless, combining the known instances of  Conjecture \ref{conj-shortzs} with Theorem \ref{thm-mult} yields many new cases where Conjecture \ref{conj-shortzs} is established here without condition. In particular, we have the following corollaries,  showing that Conjecture \ref{conj-shortzs} is true
when $n$ is only divisible by primes at most $7$, or
when $n$ is a prime power with $k\leq \frac{2n+1}{3}$, or when $n$ is composite and $k=n-d-1$ or $n-2d+1$ for a proper, nontrivial divisor $d\mid n$. The second corollary, in the case $m=1$, removes the restriction $p\nmid k$ in \cite[Theorem 5]{S_<k-invariant2/3p}.

\begin{corollary}\label{cor-mult-2,3}
 If $n=2^{s_1}3^{s_2}5^{s_3}7^{s_4}\geq 2$ with $s_1,s_2,s_3,s_4\geq 0$, then Conjecture \ref{conj-shortzs} holds  in $C_{n}\oplus C_{n}$ for all  $k\in [0,n-1]$.
\end{corollary}

\begin{corollary}\label{cor-mult}
For any prime power $n\geq 2$, Conjecture \ref{conj-shortzs} holds  in $C_{n}\oplus C_{n}$ for all  $k\leq \frac{2n+1}{3}$.
\end{corollary}

\begin{corollary}\label{cor-bigspec}
For $n\geq 4$ composite with $d\mid n$ a proper, nontrivial divisor,  Conjecture \ref{conj-shortzs} holds for $k=n-d-1$ and  for $k=n-2d+1$  in $C_{n}\oplus C_{n}$.
\end{corollary}

\section{Preparatory Lemmas}

The goal of this section is to collect together several properties about sequences having the structure given in  Conjecture \ref{conj-shortzs}. However, we will also  need the following two results. The first was a conjecture of Hamidoune established in \cite[Theorem 1]{hamconj}.

\begin{theirtheorem}\label{thm-hamconj}
Let $G$ be a finite abelian group,  let $k\geq 1$ and let $S\in \Fc(G)$ be a sequence  with $|S|\geq |G|+1$ and $k\leq |\supp(S)|$. If $\mathsf h(S)\leq |G|-k+2$ and $0\notin \Sigma_{|G|}(S)$, then $|\Sigma_{|G|}(S)|\geq |S|-|G|+k-1$.
\end{theirtheorem}

The second is  \cite[Lemma 3.2]{PropBfix}, which is  the corrected version of \cite[Proposition 4.2]{PropB}.

\begin{theirtheorem}\label{thm-fixedprop}
Let $n\geq 2$, let $s\geq 3$ and   let $G=C_n\oplus C_n$. If   $S\in\Fc(G)$  is a zero-sum sequence with $|S|=sn-1$ and $0\notin \Sigma_{\leq n-1}(S)$, then  there is a basis $(e_1,e_2)$ for $G$ such that either
\begin{itemize}
\item[1.] $\supp(S)\subseteq \{e_1\}\cup \big(\la e_1\ra+e_2\big)$ and $\vp_{e_1}(S)\equiv -1\mod n$, or
  \item[2.] $S=e_1^{[an]}\bdot e_2^{[bn-1]}\bdot (xe_1+e_2)^{[cn-1]} \bdot (xe_1+2e_2)$ for some $x\in [2,n-2]$ with $\gcd(x,n)=1$, and some  $a,b,c\geq 1$ with $a+b+c=s$.
\end{itemize}
\end{theirtheorem}

We begin now with a stability result for sequences satisfying Conjecture \ref{conj-shortzs} when $k\geq 1$.

\begin{lemma}\label{lem-perturb-solo}
Let $n\geq 2$, let $k\in [1,n-1]$,  let $G=C_n\oplus C_n$, and let $S\in \Fc(G)$ with  $|S|=2n-2+k$ and $0\notin \Sigma_{\leq 2n-1-k}(S)$ such that  Conjecture \ref{conj-shortzs} holds for $S$. If $x\in \supp(S)$, $y\in G$, and $S'=S\bdot x^{[-1]}\bdot y$ also has $0\notin \Sigma_{\leq 2n-1-k}(S')$ with  Conjecture \ref{conj-shortzs} holding for $S'$, then $x=y$.
\end{lemma}

\begin{proof}
 If $k=1$, then $S$ and $S'$ satisfying the conclusion of Conjecture \ref{conj-shortzs} implies they are both zero-sum sequences, which forces $x=y$.
If $n=2$, then $k=1\in [1,n-1]$ is forced. If $k\in [2,n-2]$, then $n\geq 4$ and $S=e_1^{[n-1]}\bdot e_2^{[n-1]}\bdot (e_1+e_2)^{[k]}$ with $e_1+(e_1+e_2)\neq e_2$ and $e_2+(e_1+e_2)\neq e_1$ in view of $n\geq 3$. Since $n\geq 3$ and $k\geq 2$, we also guaranteed $e_1,e_2,e_1+e_2\in \supp(S\bdot x^{[-1]})$. Consequently, since $S'$ also satisfies the conclusion of Conjecture \ref{conj-shortzs}, it must do so with respect to the basis $(e_1,e_2)$, forcing $x=y$. Finally, if $k=n-1$ and $n\geq 3$, then $S=e_1^{[n-1]}\bdot e_2^{[n-1]}\bdot e_3^{[n-1]}$ and $\supp(S)=\supp(S\bdot x^{[-1]})\subseteq \supp(S')$ in view of $n\geq 3$. Thus, since $S'$ also satisfies the conclusion of Conjecture \ref{conj-shortzs}, it must do so with $\supp(S')=\supp(S)$, forcing $x=y$.
\end{proof}

We continue by  showing how Property A implies the more detailed structure given in Conjecture \ref{conj-shortzs}.

\begin{lemma}\label{lem-genPropB}
Let $n\geq 4$, let $k\in [2,n-2]$, let $G=C_n\oplus C_n$, and let $S\in \Fc(G)$ be a sequence with $|S|=2n-2+k$ and $0\notin \Sigma_{\leq 2n-1-k}(S)$. Suppose there are $e_1,\,e_2\in G$ with $\supp(S)\subseteq \{e_1\}\cup \big(\la e_1\ra+e_2\big)$. Then there is some $f_2\in \la e_1\ra+e_2$ such that $(e_1,f_2)$ is a basis for $G$ and $S=e_1^{[n-1]}\bdot f_2^{[n-1]}\bdot (e_1+f_2)^{[k]}$.
\end{lemma}

\begin{proof}
By hypothesis, $\supp(S)\subseteq \{e_1\}\cup\big(\la e_1\ra+e_2\big)\subseteq \la e_1,e_2\ra$. Let $G'=\la e_1,e_2\ra\cong C_{m_1}\oplus C_{m_2}$ with $m_1\mid m_2$. Let $k'=k+2n-m_1-m_2\geq k$.
If $G'=\la e_1,e_2\ra$ were a proper subgroup, then the hypotheses $|S|=2n-2+k=m_1+m_2-1+(k'-1)$ ensures that $S$ contains a nontrivial zero-sum with length at most $\max\{m_1+m_2-1-(k'-1),m_2\}=\max\{2(m_1+m_2-n)-k,m_2\}
\leq \max\{2n-1-k,n\}=2n-1-k$, contradicting the hypothesis $0\notin \Sigma_{\leq 2n-1-k}(S)$. Therefore $G'=\la e_1,e_2\ra=G$, implying that $(e_1,e_2)$ is a basis for $G$.

In view of our hypotheses, we have $S=e_1^{[\ell]}\bdot \prod^\bullet_{i\in [1, 2n-2+k-\ell]}(x_ie_1+e_2)$ for some $\ell\geq 0$ and $x_i\in [0,n-1]$. We must have $\ell\leq n-1$, else $S$ will contain an $n$-term zero-sum, contrary to the hypothesis  $0\notin \Sigma_{\leq 2n-1-k}(S)$. Let $S_2=\prod^\bullet_{i\in [1, 2n-2+k-\ell]}x_ie_1$ and $S_1=e_1^{[\ell]}$. Then $|S_2|=2n-2+k-\ell\geq n-1+k\geq n+1$. We also have $\mathsf h(S_2)\leq n-1$, else $S$ again contains an $n$-term zero-sum, contrary to hypothesis. Thus $|\supp(S_2)|\geq 2$.

Suppose $|S_1|=\ell\leq n-1-k$. Then the hypothesis $0\notin \Sigma_{\leq 2n-1-k}(S)$ implies $0\notin \Sigma_n(S_2)+(\Sigma(S_1)\cup \{0\})$, whence $\Sigma_{n}(S_2)\subseteq [1,n-1-\ell]_{e_1}$. In particular, $|\Sigma_n(S_2)|\leq n-1-\ell$. However, applying Theorem \ref{thm-hamconj} to $S_2$ (using $k=2$), we obtain $|\Sigma_n(S_2)|\geq |S_2|-n+1=n-1+k-\ell>n-1-\ell$, contradicting what was just noted. So we can now assume $|S_1|=\ell\geq n-k$.

Since $|S_1|=\ell\geq n-k$, the hypothesis $0\notin \Sigma_{\leq 2n-1-k}(S)$ implies that $0\notin \Sigma_n(S_2)+(\Sigma_{\leq n-k-1}(S_1)\cup \{0\})$, whence $\Sigma_{n}(S_2)\subseteq [1,k]_{e_1}$. In particular, $|\Sigma_n(S_2)|\leq k$. Applying Theorem \ref{thm-hamconj} to $S_2$ (using $k=2$), and then using  the estimate $\ell\leq n-1$,  we obtain $|\Sigma_n(S_2)|\geq |S_2|-n+1=n-1+k-\ell\geq k$. Thus equality must hold in all these estimates. In particular,  $\Sigma_n(S_2)=[1,k]_{e_1}$, $\ell=n-1$, and $|\Sigma_n(S_2)|= |S_2|-n+1$. It now follows from Theorem \ref{thm-hamconj} applied to $S_2$ (using $k=3$) that $|\supp(S_2)|=2$.

Let $ye_1\in \supp(S_2)$ be an element with maximum multiplicity in $S_2$, and let $f_2=ye_1+e_2$. Then $(e_1,f_2)$ is also a basis for $G$ and  \be\label{structure}S=e_1^{[n-1]}\bdot f_2^{[n-1-r]}\bdot (xe_1+f_2)^{k+r}\ee for some $x\in [1,n-1]$ and $r\in [0,\frac{n-1-k}{2}]$.
Let $S'_2=0^{[n-1-r]}\bdot (xe_1)^{[k+r]}$. Repeating the argument of the previous paragraph using $S'_2$ in place of $S_2$, we again conclude that $\Sigma_n(S'_2)=[1,k]_{e_1}$. However, in view of the structure of $S$ given by \eqref{structure}, we have $\Sigma_n(S'_2)=(r+1)xe_1+[0,k-1]_{xe_1}$.
Thus \be\label{equality}[1,k]_{e_1}=(r+1)xe_1+[0,k-1]_{xe_1}.\ee Since $k\geq 2$, the set $[1,k]_{e_1}$ is not contained in a coset of a proper subgroup of $\la e_1\ra$. Hence \eqref{equality} ensures $\la xe_1\ra=\la e_1\ra$. The left-hand side of \eqref{equality} is an arithmetic progression with difference $e_1$ and length $k$, with $2\leq k\leq n-2=\ord(e_1)-2$. It is well known and easily derived that, for such sets, the difference $e_1$ is unique up to sign. The right-hand side of \eqref{equality} is also an arithmetic progression with difference $xe_1$ and length $k$, with $2\leq k\leq n-2=\ord(xe_1)-2$. Thus, by the uniqueness of the difference, it follows that $xe_1=\pm e_1$.

If $xe_1=e_1$,  then \eqref{equality} forces $r=0$ in view of $k<n$, yielding the desired structure for $S$. If $xe_1=-e_1$, then \eqref{equality} forces $r=n-k-1$ in view of $k<n$.  However, since $r\in [0,\frac{n-1-k}{2}]$, this in only possible if $k\geq n-1$, which is contrary to hypothesis.
\end{proof}

The following lemma shows that the extension of a sequence satisfying Conjecture \ref{conj-shortzs}, obtained by concatenating an additional term, also satisfies Conjecture \ref{conj-shortzs}.

\begin{lemma}\label{lem-k-to-k+1}
Let $n\geq 2$, let $k\in [1,n-1]$ with either $k=1$ or $k\in [1,n-2]$,  let $G=C_n\oplus C_n$, and let $S\in \Fc(G)$ be a sequence with  $|S|=2n-2+k$ and $0\notin \Sigma_{\leq 2n-1-k}(S)$ such that  Conjecture \ref{conj-shortzs} holds for $S$. Suppose  there is some $g\in G$ such that $0\notin \Sigma_{\leq 2n-2-k}(S\bdot g)$.
Then there exists a basis $(e_1,e_2)$ for $G$ such that $S\bdot g=e_1^{[n-1]}\bdot e_2^{[n-1]}\bdot (e_1+e_2)^{[k+1]}$ with $g=e_1+e_2$. In particular, Conjecture \ref{conj-shortzs} holds for $S\bdot g$ (for $k\leq n-2$).
\end{lemma}

\begin{proof}
Let $(e_1,e_2)$ be an arbitrary basis for which Conjecture \ref{conj-shortzs} holds for $S$. Let $g=x_1e_1+x_2e_2$ with $x_1,\, x_2\in [0,n-1]$.

\subsection*{Case 1:}  $k\in [2,n-2]$.

In this case,  $n\geq 4$ and $S=e_1^{[n-1]}\bdot e_2^{[n-1]}\bdot (e_1+e_2)^{[k]}$. By symmetry, we can w.l.o.g. assume $x_1\geq x_2$. If $x_2=0$, then $S\bdot g$ contains $n$ terms from $\la e_1\ra\cong C_n$, and thus contains a zero-sum subsequence of length at most $\mathsf D(C_n)=n$, contradicting that $0\notin \Sigma_{2n-2-k}(S\bdot g)$ (in view of $k\leq n-2$). Therefore $x_1\geq x_2\geq 1$.
If $x_1=x_2=1$, the desired conclusion follows, so we can  assume $x_1\geq 2$. If $x_1\geq n-k$, then $e_2^{[x_1-x_2]}\bdot (e_1+e_2)^{[n-x_1]}\bdot (x_1e_1+x_2e_2)$ is a zero-sum subsequence of $S\bdot g$ with length $n-x_2+1\leq n$, contradicting that $0\notin \Sigma_{2n-2-k}(S\bdot g)$. On the other hand, if $x_2\leq x_1\leq n-k$, then $e_1^{[n-k-x_1]}\bdot e_2^{[n-k-x_2]}\bdot (e_1+e_2)^{[k]}\bdot (x_1e_1+x_2e_2)$ is a zero-sum subsequence of $S\bdot g$ with length $2n-k-x_1-x_2+1\leq 2n-2-k$ (with the latter inequality in view of $x_1\geq 2$ and $x_2\geq 1$), again contradicting that $0\notin \Sigma_{2n-2-k}(S\bdot g)$.

\subsection*{Case 2:} $k=1$.

In this case, $n\geq 2$ and  $$S=e_1^{[n-1]}\bdot {\prod}^\bullet_{i\in [1,n]}(y_ie_1+e_2)$$ for some $y_1,\ldots,y_n\in [0,n-1]$ with $y_1+\ldots+y_n\equiv 1\mod n$.
If $n=2$, then $S=e_1\bdot e_2\bdot (e_1+e_2)$, and our hypothesis   $0\notin \Sigma_{\leq 2n-2-k}(S\bdot g)=\Sigma_{\leq 1}(S\bdot g)$ simply means $g\neq 0$. In this case, replacing the basis $(e_1,e_2)$ by a basis $(f_1,f_2)$ with $g\notin\{f_1,f_2\}$, we find $S=f_1\bdot f_2\bdot (f_1+f_2)$ with $g=f_1+f_2$, and the desired result follows. Therefore we now assume $n\geq 3$, so $k=1\leq n-2$.

If $x_2=0$, then $S\bdot g$ contains $n$ terms from $\la e_1\ra\cong C_n$, and thus contains a zero-sum subsequence of length at most $\mathsf D(C_n)=n$, contradicting that $0\notin \Sigma_{2n-2-k}(S\bdot g)$ (in view of $k\leq n-2$). Therefore $x_2\geq 1$.

Let $S_2=\prod_{i\in [1,n]}^\bullet y_ie_1$. For any $(-x_1+z)e_1\in \Sigma_{n-x_2}(S_2)$, where $z\in [1,n]$, we have a subset $I\subseteq [1,n]$ with  $|I|=n-x_2$ and  $\Summ{i\in I}(y_ie_1+e_2)=(-x_1+z)e_1+(n-x_2)e_2$, meaning $e_1^{[n-z]}\bdot (x_1e_1+x_2e_2)\bdot \prod_{i\in I}^\bullet (y_ie_1+e_2)$ is a zero-sum subsequence of $S\bdot g$ of length $2n-z+1-x_2$. Since $0\notin \Sigma_{\leq 2n-2-k}(S\bdot g)=\Sigma_{2n-3}(S\bdot g)$, this forces \be\label{morg}z+x_2\leq 3.\ee Now $S_2$ is a sequence of $n$ terms from a cyclic group of order $n$ with $n-x_2\in [1,n-1]$. Moreover, since $y_1+\ldots+y_n\equiv 1\mod n$, we have $|\supp(S_2)|\geq 2$.

If $T\mid S_2$ is any subsequence of length $n-x_2$, then $n-x_2\in [1,n-1]=[1,|S_2|-1]$ ensures that both $T$ and $T^{[-1]}\bdot S_2$ contain at least one term, and since $|\supp(S_2)|\geq 2$, it is thus possible to find terms $g\in \supp(T)$ and $h\in \supp(T^{[-1]}\bdot S_2)$ with $g\neq h$. This ensures that $T\bdot g^{[-1]}\bdot h$ is also a subsequence of $S_2$ with length $n-x_2$, and one with sum $\sigma(T)-g+h\neq \sigma(T)$. Thus  $|\Sigma_{n-x_2}(S_2)|\geq 2$, meaning it is possible to find $I$ as defined above with $z\geq 2$. Combined with \eqref{morg} and $x_2\geq 1$, it follows that only $x_2=1$ is possible, whence $g=x_1e_1+e_2$.
Since $(e_1,xe_1+e_2)$ is also a  basis for which Conjecture \ref{conj-shortzs} holds for $S$, for any $x\in \Z$, we can replace the arbitrary basis $(e_1,e_2)$ for which Conjecture \ref{conj-shortzs} holds for $S$ with the basis $(e_1,(x_1-1)e_1+e_2)$, thereby allowing us to w.l.o.g. assume $x_1=1$ in view of  $e_1+\Big((x_1-1)e_1+e_2\Big)=g$. Thus we now have $g=e_1+e_2$ with $x_1=x_2=1$.

Since $x_2=1$, we have $$\Sigma_{n-x_2}(S_2)=\Sigma_{n-1}(S_2)=\sigma(S_2)-\Sigma_1(S_2)=e_1-\supp(S_2),$$ with the final inequality above in view of $y_1+\ldots+y_n\equiv 1\mod n$. Since $x_2=1$, \eqref{morg} ensures that $z\leq 2$, which combined with $z\in [1,n]$ forces $z\in \{1,2\}$. Thus $$e_1-\supp(S_2)=\Sigma_{n-x_2}(S_2)\subseteq \{-x_1e_1+e_1,-x_1e_1+2e_1\}=\{0,e_1\},$$ whence $$\supp(S_2)=\{0,e_1\}$$ in view of $|\supp(S_2)|\geq 2$. It follows that $y_i\equiv 0$ or $1\mod n$ for every $i\in [1,n]$. Letting $a\in [1,n-1]$ be the number of $i\in [1,n]$ with $y_i\equiv 1\mod m$, we find $1\equiv y_1+\ldots+y_n\equiv a+(n-a)(0)\mod n$, implying $a\equiv 1\mod n$. Thus $S=e_1^{[n-1]}\bdot e_2^{[n-1]}\bdot (e_1+e_2)$ with $g=e_1+e_2$, as desired.
\end{proof}

The following lemma is the reverse of Lemma \ref{lem-k-to-k+1}, showing that, if Conjecture \ref{conj-shortzs} holds for a sequence and we remove a term, then Conjecture \ref{conj-shortzs} also holds for the resulting subsequence.

\begin{lemma}\label{lem-k+1-to-k}
Let $n\geq 3$, let $k\in [1,n-2]$,  let $G=C_n\oplus C_n$, and let $S\in \Fc(G)$ be a sequence with  $|S|=2n-2+k$ and $0\notin \Sigma_{\leq 2n-1-k}(S)$. Suppose  there is some $g\in G$ such that $0\notin \Sigma_{\leq 2n-2-k}(S\bdot g)$ with Conjecture \ref{conj-shortzs} holding for $S\bdot g$.
Then there exists a basis $(e_1,e_2)$ for $G$ such that $S\bdot g=e_1^{[n-1]}\bdot e_2^{[n-1]}\bdot (e_1+e_2)^{[k+1]}$ with $g=e_1+e_2$. In particular, Conjecture \ref{conj-shortzs} holds for $S$.
\end{lemma}

\begin{proof}
Let $(e_1,e_2)$ be an arbitrary basis for which Conjecture \ref{conj-shortzs} holds for $S\bdot g$. Then $$S\bdot g=e_1^{[n-1]}\bdot e_2^{[n-1]}\bdot (xe_1+e_2)^{[k+1]}$$ for some $x\in [1,n-1]$ with $\gcd(x,n)=1$ and either $x=1$ or $k=n-2$.

Suppose $x=1$. In such case, if $g=e_1+e_2$, the proof is complete, so either $g=e_2$ or $g=e_1$. But now $(e_1+e_2)^{[k+1]}\bdot e_1^{[n-k-1]}\bdot e_2^{[n-k-1]}$ is a zero-sum subsequence of $S$ (in view of the hypothesis $k\geq 1$) with length $2n-1-k$, contradicting that $0\notin \Sigma_{\leq 2n-1-k}(S)$. So we can now assume $x\geq 2$ with $k=n-2$, in which case  $$S\bdot g=e_1^{[n-1]}\bdot e_2^{[n-1]}\bdot (xe_1+e_2)^{[n-1]}.$$

If $x=n-1$, then using the basis $(e_1, -e_1+e_2)$ in place of $(e_1,e_2)$, we find ourselves in the already completed case when $x=1$. Thus we can assume $x\in [2,n-2]$ with $\gcd(x,n)=1$, implying $n\geq 5$. Thus $k=n-2\geq 3$ with $\supp(S)\subseteq \{e_1\}\cup (\la e_1\ra+e_2)$, $|S|=2n-2+k$ and  $0\notin \Sigma_{\leq 2n-1-k}(S)$, allowing us to apply Lemma \ref{lem-genPropB} to $S$ to conclude that there is a basis $(f_1,f_2)$ for $G$ such that $S=f_1^{[n-1]}\bdot f_2^{[n-1]}\bdot (f_1+f_2)^{[k]}$. Since every term of $S\bdot g$ has multiplicity $n-1$, it follows that $g=f_1+f_2$, and the desired conclusion follows.
\end{proof}

\section{The Main Proof}

We divide the proof of Theorem \ref{thm-mult} into two main cases depending on the value of $k_n\in [0,n-1]$.  We begin first with the case when $k_n\in [0,1]$.

\begin{proposition}\label{prop-main-k_n-01}
Let $m,\,n\geq 2$ and let $k\in [0,mn-1]$ with   $k=k_mn+k_n$, where  $k_m\in [0,m-1]$ and $k_n\in [0,1]$. Suppose either Conjecture \ref{conj-shortzs} holds for $k_m$ in $C_m\oplus C_m$, or else $k_n=1$, $k_m\in [1,m-2]$ and Conjecture \ref{conj-shortzs} holds for $k_m+1$ in $C_m\oplus C_m$. Then Conjecture \ref{conj-shortzs} holds for $k$ in $C_{mn}\oplus C_{mn}$.
\end{proposition}

\begin{proof}
As remarked in the introduction, Conjecture \ref{conj-shortzs}  holds for $k\leq 1$ or $k=mn-1$ in every group $C_{mn}\oplus C_{mn}$. Therefore we can assume $k_m\in [1,m-1]$ and   $k=k_mn+k_n\in [2,mn-2]$.
 Let $G=C_{mn}\oplus C_{mn}$ and let $S\in \Fc(G)$ be a sequence with \be\label{start-hyp}|S|=2nm-2+k\quad\und\quad 0\notin \Sigma_{\leq 2nm-1-k}(S).\ee We need to show Conjecture \ref{conj-shortzs}.3 holds for $S$. Let $\varphi:G\rightarrow G$ be the multiplication by $m$ homomorphism, so $\varphi(x)=mx$. Note $$\varphi(G)=mG\cong C_n\oplus C_n\quad\und\quad \ker \varphi=nG\cong C_m\oplus C_m.$$

If $k_n=1$, set $S^*=S$.  If $k_n=0$, we can choose any element $g_0\in -\sigma(S)+\ker \varphi$ and set $S^*=S\bdot g_0$. When $k_n=0$, the definition of $g_0$ ensures that $\varphi(S^*)$ is zero-sum. When $k_n=1$, we will shortly see below in Claim A that $\varphi(S^*)$ is also zero-sum. Note, in all cases, $$|S^*|=2mn-1+k_mn.$$
Define a \emph{block decomposition} of $S^*$ to be a factorization $$S^*=W_0\bdot W_1\bdot\ldots\bdot W_{2m-2+k_m}$$ with $1\leq |W_i|\leq n$ and $\varphi(W_i)$ zero-sum for each $i\in [1,2m-2+k_m]$. Since $\mathsf s_{\leq n}(\varphi(G))=\mathsf s_{\leq n}(C_n\oplus C_n)=3n-2$ and $|S|\geq (2m-3+k_m)n+3n-2$, it follows by repeated application of the definition of $\mathsf s_{\leq n}(\varphi(G))$ that $S^*$ has a block decomposition, and one  with $g_0\in\supp(W_0)$ when  $k_n=0$.

 \subsection*{Claim A} If $S^*=W_0\bdot W_1\bdot\ldots\bdot W_{2m-2+k_m}$ is a block decomposition, then $|W_i|=n$ for all $i\in [1,2m-2+k_m]$, \ $\varphi(W_0)$ is a minimal zero-sum sequence of length  $|W_0|=2n-1$, and $0\notin \Sigma_{\leq n-1}(\varphi(S^*))$.

 \begin{proof}
Suppose $k_n=1$, so $S^*=S$.
Let us show that $0\notin \Sigma_{\leq 2n-2}(\varphi(W_0))$. Assuming this fails, there is  a nontrivial subsequence $W'_0\mid W_0$ with $|W'_0|\leq 2n-2$ and $\varphi(W'_0)$ zero-sum. Set $W'_i=W_i$ for $i\in [1,2m-2+k_m]$. Then $S_\sigma=\sigma(W'_0)\bdot \sigma(W'_1)\bdot\ldots\bdot \sigma(W'_{2m-2+k_m})$ is a sequence of terms from $\ker\varphi\cong C_m\oplus C_m$ with $|S_\sigma|=2m-1+k_m$. Since $\mathsf s_{\leq 2m-1-k_m}(C_m\oplus C_m)=2m-1+k_m$, it follows that $S_\sigma$ has a nontrivial zero-sum subsequence of length at most $2m-1-k_m$, say $\prod^\bullet_{i\in I}\sigma(W'_i)$ for some nonempty subset $I\subseteq [0,2m-2+k_m]$ with $|I|\leq 2m-1-k_m$.
But now $\prod_{i\in I}^\bullet W'_i$ is a nontrivial  zero-sum subsequence of $S^*=S$ with length $$\Summ{i\in I}|W'_i|\leq \max\{|W'_0|,\,n\}+(|I|-1)n\leq 2n-2+(2m-2-k_m)n=2mn-1-k,$$ contradicting \eqref{start-hyp}.
This show that $0\notin \Sigma_{2n-2}(\varphi(W_0))$.
As a result, $$|W_0|\leq \mathsf s_{\leq 2n-2}(\varphi(G))-1=\mathsf s_{\leq 2n-2}(C_n\oplus C_n)-1=2n-1.$$

Suppose $k_n=0$. Then $S^*=S\bdot g_0$ with  $\varphi(S\bdot g_0)$ zero-sum  by definition of $g_0$.  Hence $\varphi(W_0)$ is also a zero-sum sequence. Let us show that $\varphi(W_0)$ is a minimal zero-sum sequence. Assuming this fails, then $W_0$ contains disjoint, nontrivial subsequences $W_{2m-1+k_m}\bdot W_{2m+k_m}\mid W_0$ with $|W_{2m-1+k_m}|+|W_{2m+k_m}|\leq n+2n-1$  and $\varphi(W_{2m-1+k_m})$ and $\varphi(W_{2m+k_m})$ both zero-sum (if $|W_0|\leq 3n-1$, this is trivial in view of $\varphi(W_0)$ not being a minimal zero-sum, while the same conclusion follows from  $\mathsf D(\varphi(G))=2n-1$ and $\mathsf s_{\leq n}(\varphi(G))=3n-2$ when $|W_0|\geq 3n-1$).
By passing to appropriate zero-sum subsequences,  we can then further assume $\varphi(W_{2m-1+k_m})$ and $\varphi( W_{2m+k_m})$ are each minimal zero-sum subsequences, so that $|W_i|\leq \mathsf D(\varphi(G))=2n-1$ for both $i\in \{2m-1+k_m, 2m-k_m\}$.
As at most one of the sequences $W_j$ can contain the term $g_0$, it follows that $\prod_{i\in [1,2m+k_m]\setminus \{j\}}^\bullet W_j\mid S$ for some $j\in [1,2m+k_m]$.
Now $\sigma(W_1)\bdot\ldots\bdot \sigma(W_{2m+k_m})\bdot \sigma(W_{j})^{[-1]}$ is a sequence of terms from  $\ker \varphi\cong C_m\oplus C_m$ with length $\mathsf s_{\leq 2m-1-k_m}(C_m\oplus C_m)=2m-1+k_m$.
It follows that there is a zero-sum subsequence  $\prod^\bullet_{i\in I}\sigma(W_i)$ for some $I\subseteq [1,2m+k_m]\setminus \{j\}$ with $1\leq |I|\leq 2m-1-k_m$.
In such case,  if $|I|\geq 2$, then $T=\prod^\bullet_{i\in I}W_i$ is a zero-sum subsequence of $S$ with length \begin{align*}|T|&\leq (|I|-2)n+\max\{2n,n+|W_{2m-1+k_m}|, n+|W_{2m+k_m}|,|W_{2m-1+k_m}|+|W_{2m+k_m}|\}\\&\leq (|I|-2)n+3n-1
\leq (2m-3-k_m)n+3n-1=2mn-1+k,\end{align*} contrary to \eqref{start-hyp}.
On the other hand, if $|I|=1$, then $T=\prod^\bullet_{i\in I}W_i$ is a zero-sum subsequence of $S$ with length $|T|\leq \max\{n,|W_{2m-1+k_m}|,|W_{2m+k_m}|\}\leq 2n-1\leq 2mn-1-k$, also contradicting \eqref{start-hyp}.
This shows that $\varphi(W_0)$ must be a minimal zero-sum sequence. In particular, $$|W_0|\leq \mathsf D(C_n\oplus C_n)=2n-1.$$

Regardless of whether $k_n=0$ or $1$, we have shown that $|W_0|\leq 2n-1$.
As a result, since $|W_i|\leq n$ for all $i\in [1,2m-2+k_m]$, we have
$$2n-1=2mn-1+k_mn-(2m-2+k_m)n\leq |S^*|-\Sum{i=1}{2m-2+k_m}|W_i|=|W_0|\leq 2n-1,$$ forcing equality to hold in these estimates, i.e., $|W_i|=n$ for all $i\in [1,2m-2+k_m]$ and $|W_0|=2n-1$. If $k_n=0$, we have already shown that $\varphi(W_0)$ is a minimal zero-sum sequence. For $k_n=1$, we established that $0\notin \Sigma_{2n-2}(\varphi(W_0))$, which combined with  $\mathsf D(\varphi(G))=\mathsf D(C_n\oplus C_n)=2n-1$ forces $\varphi(W_0)$ to be  a minimal zero-sum sequence in this case as well.
If $0\in \Sigma_{\leq n-1}(\varphi(S^*))$, then  there is a nontrivial subsequence $W'_1\mid S^*$ with $1\leq |W'_1|\leq n-1$ and $\varphi(W'_1)$ zero-sum. Then, by the argument showing that $S^*$ has some block decomposition, we can find a block decomposition $S^*=W'_0\bdot W'_1\bdot\ldots \bdot W'_{2m-2+k_m}$ with $|W'_1|<n$, contrary to what was just established for an arbitrary block decomposition. Thus $0\notin \Sigma_{\leq n-1}(\varphi(S^*))$,
and all parts of Claim A are established.
 \end{proof}

Suppose \be\label{weak-deff}S^*= W_0\bdot\ldots\bdot W_{2m-2+k_m}\ee  with each $\varphi(W_i)$ a nontrivial zero-sum for $i\in [0,2m-2+k_m]$. We call this a \emph{weak block decomposition} of $S^*$. In view of Claim A, we have $|W_i|\geq n$ for all $\in [0,2m-2+k_m]$, and since $|S^*|=2mn-1-k_mn>(2m-1-k_m)n$, we cannot have $|W_i|=n$ for all $i\in [0,2m-2+k_m]$.
$$\mbox{Let $k_\emptyset\in [0,2m-2+k_m]$ be an index with }\quad  \left\{
               \begin{array}{ll}
                |W_{k_\emptyset}|>n & \hbox{if $k_n=1$,} \\
                 g_0\in \supp(W_0) & \hbox{if $k_n=0$.}
               \end{array}
             \right.$$
 Then define
$$S_\sigma=\sigma(W_0)\bdot\ldots\bdot \sigma(W_{2m-2+k_m})\bdot \sigma(W_{k_\emptyset})^{[-1]}\in \Fc(\ker \varphi).$$
We call $k_\emptyset$ and $S_\sigma$ the \emph{associated} index and sequence for the block decomposition.
For $j\in [0,2m-2+k_m]$, set $$\wtilde W_j=\left\{
                    \begin{array}{ll}
                      W_j\bdot g_0^{[-1]} & \hbox{ if $k_n=0$ and $j=k_\emptyset$;} \\
                      W_j & \hbox{otherwise.}
                    \end{array}
                  \right.$$
In view of Claim A, any block decomposition is also a weak block decomposition. If $S^*=W_0\bdot W_1\bdot\ldots\bdot W_{2m-2+k_m}$ is a block decomposition and $k_n=1$, then $k_\emptyset=0$ is forced as $|W_i|=n$ for all $i\geq 1$. On the other hand, if $k_n=0$, then there is a block decomposition with $g_0\in \supp(W_0)$ as remarked earlier, and thus with $k_\emptyset=0$.

\subsection*{Claim B} Suppose $S^*= W_0\bdot\ldots\bdot W_{2m-2+k_m}$ is a weak block decomposition with associated index  $k_\emptyset$ and associated sequence $S_\sigma$.  Then $|S_\sigma|=2m-2+k_m$ with $0\notin \Sigma_{\leq 2m-1-k_m}(S_\sigma)$.
 Moreover, if we also have $k_n=1$, then $|\sigma(W_{k_\emptyset})\bdot S_\sigma|=2m-1+k_m$ with $0\notin \Sigma_{\leq 2m-2-k_m}(\sigma(W_{k_\emptyset})\bdot S_\sigma)$. Regardless of whether $k_n=0$ or $1$, Conjecture \ref{conj-shortzs} holds for  $S_\sigma$.

\begin{proof} We have $|S_\sigma|=2m-2+k_m$ by definition. Assume by contradiction $0\in \Sigma_{\leq 2m-1-k_m}(S_\sigma)$. Then  there is a  zero-sum subsequence $\prod_{i\in I}^\bullet\sigma(W_i)$ for some $I\subseteq [0,2m-2+k_m]\setminus \{k_\emptyset\}$ with $1\leq |I|\leq 2m-1-k_m$. By Claim A, we have $0\notin \Sigma_{\leq n-1}(\varphi(S^*))$, which ensures $|W_i|\geq n$ for all $i\in [0,2m-2+k_m]\setminus I$.
 Since $k_\emptyset\notin I$, we have $|W_{k_\emptyset}|\geq n+k_n$ with $k_\emptyset\in [0,2m-2+k_m]\setminus I$ (by definition of $k_\emptyset$).
 It follows that $T:=\prod_{i\in I}^\bullet W_i$ is a nontrivial zero-sum subsequence of $S$ with length $|T|=|S^*|-\Summ{i\in [0,2m-2+k_m]\setminus I}|W_i|\leq |S^*|-(2m-1+k_m-|I|)n-k_n\leq |S^*|-2k_mn-k_n=2mn-1-k$, contradicting \eqref{start-hyp}. So we instead conclude that $|S_\sigma|=2m-2+k_m$ and $0\notin \Sigma_{\leq 2m-1-k_m}(S_\sigma)$.

 Suppose $k_n=1$, so that $S^*=S$. Assume by contradiction that $0\in \Sigma_{\leq 2m-2-k_m}(\sigma(W_{k_\emptyset})\bdot S_\sigma)$. Then  there is a  zero-sum subsequence $\prod_{i\in I}^\bullet\sigma(W_i)$ for some $I\subseteq [0,2m-2+k_m]$ with $1\leq |I|\leq 2m-2-k_m$. By Claim A, we have $0\notin \Sigma_{\leq n-1}(\varphi(S^*))$, which ensures $|W_i|\geq n$ for all $i\in [0,2m-2+k_m]\setminus I$.
 Hence  $T:=\prod_{i\in I}^\bullet W_i$ is a nontrivial zero-sum subsequence of $S$ with length \begin{align*}|T|&=|S^*|-\Summ{i\in [0,2m-2+k_m]\setminus I}|W_i|\leq |S^*|-(2m-1+k_m-|I|)n\leq |S^*|-(2k_m+1)n\\&=2mn-1-k_mn-n\leq 2mn-1-k,\end{align*} contradicting \eqref{start-hyp}.
 So we instead conclude that we have  $|\sigma(W_{k_\emptyset})\bdot S_\sigma|=2m-1+k_m$ and $0\notin \Sigma_{\leq 2m-2-k_m}(\sigma(W_{k_\emptyset})\bdot S_\sigma)$.

If Conjecture \ref{conj-shortzs} holds for $k_m$ in $C_m\oplus C_m$, then the first part of Claim B ensures that Conjecture \ref{conj-shortzs} holds for $S_\sigma$. Otherwise, the hypotheses of  Proposition \ref{prop-main-k_n-01} ensure that $k_n=1$ and  $k_m\in [1,m-2]$ with Conjecture \ref{conj-shortzs} holding for $k_m+1$ in $C_m\oplus C_m$. In such case, the second part of Claim B ensures that Conjecture \ref{conj-shortzs} holds for $\sigma(W_{k_\emptyset})\bdot S_\sigma$, and then applying Lemma \ref{lem-k+1-to-k} shows that Conjecture \ref{conj-shortzs} holds for $S_\sigma$.
%Thus, in both cases, we find that Conjecture \ref{conj-shortzs} must hold for $S_\sigma$, regardless of whether $k_n=0$ or $1$.
\end{proof}

%\subsection*{Claim C} If $k_n=0$ and $S=W_0\bdot W_1\bdot\ldots\bdot W_{2m-2+k_m}$ is a block decomposition of $S$, then $|W_i|=n$ for all $i\in [1,2m-2+k_m]$, $\varphi(W_0)$ is zero-sum free with $|W_0|=2n-2$, $0\notin \Sigma_{\leq n-1}(\varphi(S))$, and $S_\sigma=\sigma(W_1)\bdot\ldots\bdot \sigma(W_{2m-2+k_m})\in \Fc(\ker \varphi)$ is a sequence of length $|S_\sigma|=2m-2+k_m$ with $0\notin \Sigma_{\leq 2m-1-k_m}(S_\sigma)$.

%\begin{proof}
%Observing the $S^*=(W_0\bdot g_0)\bdot W_1\bdot\ldots\bdot W_{2m-2+k_m}$ is a block decomposition of $S^*$, all parts of the claim follow from Claims A and B.
%\end{proof}

\subsection*{Claim C} There exists a  basis $(f_1,f_2)$ for $\varphi(G)=C_n\oplus C_n$ such that either
 \begin{itemize}
 \item[1.] $\supp(\varphi(S^*))\subseteq f_1\cup \big(\la f_1\ra+f_2\big)$, or
 \item[2.]
 $\varphi(S^*)=f_1^{[an]}\bdot f_2^{[bn-1]}\bdot f_3^{[cn-1]}\bdot (f_2+f_3)$, where $f_3=xf_1+f_2$ for some  $x\in [2,n-2]$ with $\gcd(x,n)=1$, $a,b,c\geq 1$ and $a+b+c=2m+k_m$.
 \end{itemize}

\begin{proof}
By Claim A, we have $0\notin \Sigma_{\leq n-1}(\varphi(S^*))$, while $|\varphi(S^*)|=|S^*|=(2m+k_m)n-1$. Thus Claim C follows from Theorem \ref{thm-fixedprop}.
\end{proof}

We define a term $g\in \supp(S)$ to be \emph{good} if $g,\,h\in \supp(S)$ with $\varphi(g)=\varphi(h)$ implies $g=h$.
 A term $g\in \supp(\varphi(S))$ is \emph{good} if $\supp(S)$ contains exactly one element from $\varphi^{-1}(g)$. Then, for $g\in \supp(S)$, we find that  $\varphi(g)=mg$ is good if and only if $g$ is good.

\subsection*{Claim D} Suppose $S^*=W_0\bdot W_1\bdot\ldots\bdot W_{2m-2+k_m}$ is a weak block decomposition. If   $g\in \supp( W_j)$, \ $h\in \supp(\wtilde W_{k_\emptyset})$ and $\varphi(g)=\varphi(h)$, where $j,\,k_\emptyset\in [0,2m-2+k_m]$ are \emph{distinct}, then $g=h$ is good.

  \begin{proof}
Since $\varphi(g)=\varphi(h)$, setting $W'_j=W_j\bdot g^{[-1]}\bdot h$, $W'_{k_\emptyset}=W_{k_\emptyset}\bdot h^{[-1]}\bdot g$, and $W'_i=W_i$ for all $i\neq j,k$, we obtain a new weak block decomposition  $S^*=W'_0\bdot W'_1\bdot \ldots\bdot W'_{2m-2+k_m}$.  Since $h\in \supp(\wtilde W_{k_\emptyset})$, we have $g_0\in \supp(W'_{k_\emptyset})$ for $k_n=0$, and $|W'_{k_\emptyset}|=|W_{k_\emptyset}|>n$ for $k_n=1$. Consequently, if we let $S_\sigma$ and $S'_\sigma$ be the associated sequences for the original and new block decompositions, with $k_\emptyset$ and $k'_
\emptyset$ the associated indices, we find that $k_\emptyset=k'_\emptyset$ with  $S'_{\sigma}$ obtained from $S_\sigma$ by replacing the term $\sigma(W_j)$ by the term $\sigma(W'_j)=\sigma(W_j)-g+h$. In view of Claim B, it follows that Conjecture \ref{conj-shortzs} holds for both sequences $S'_{\sigma}$ and $S_\sigma$ using $k_m\in [1,m-1]$ modulo $m$. Thus Lemma \ref{lem-perturb-solo} implies that $g=h$.
Repeating this argument for an arbitrary $g'\in \supp(S\bdot W_{k_\emptyset}^{[-1]})$ using the fixed $h\in \supp(\wtilde W_{k_\emptyset})$, we conclude that $g'=h=g$ for all $g'\in \supp(S\bdot W_{k_\emptyset}^{[-1]})$ with $\varphi(g')=\varphi(g)=\varphi(h)$.
Likewise, repeating the argument for an arbitrary $h'\in\supp(\wtilde W_{k_\emptyset})$ using the fixed $g\in  \supp(W_j)$, we find $h'=g=h$ for all $h'\in \supp(\wtilde W_{k_\emptyset})$ with $\varphi(h')=\varphi(g)=\varphi(h)$. It follows that $g=h$ is good.
  \end{proof}

\subsection*{Claim E} Claim C.1 holds for $S^*$.

\begin{proof}
Assume instead that Claim C.2 holds.
Let $(f_1,f_2)$ be a basis for which Claim C.2 holds and let $x^*\in [2,n-2]$ be the multiplicative inverse of $-x$ modulo $n$, so $$x^*x\equiv -1\mod n.$$
In view of Claim C.2, there is a block decomposition $S^*=W_0\bdot W_1\bdot\ldots \bdot W_{2m-2+k_m}$ with $\varphi(W_0)=f_1^{[n-1]}\bdot f_2^{[x^*]}\bdot f_3^{[n-x^*]}$, $\varphi(W_1)=f_3^{[x^*-1]}\bdot f_2^{[n-x^*-1]}\bdot f_1\bdot (f_2+f_3)$ and $\varphi(W_i)\in \{f_1^{[n]}, f_2^{[n]}, f_3^{[n]}\}$ for all $i\in [2,2m-2+k_m]$. We call any such block decomposition a \emph{strong} block decomposition of $S^*$.
For $k_n=0$, it can also be assumed that either $g_0\in \supp(W_0)$ or else $g_0\in \supp(W_1)$ and $\varphi(g_0)=f_2+f_3$, since  $\supp(\varphi(W_0))$ contains every element from $\supp(\varphi(S^*))$ apart from the unique term equal to $f_2+f_3$ which is contained in $\varphi(W_1)$.

Let us first show all $g\in \supp(S)$ are good.
Since Claim C.2 holds, we have $n\geq 5$ (as $x\in [2,n-2]$ with $\gcd(x,n)=1$) and there is a unique term $g$ of $S^*$ with $\varphi (g)=f_2+f_3$, which is  trivially good if $g\in \supp(S)$. Consider $f\in \{f_1,f_2,f_1+f_2\}$ and let $S^*=W_0\bdot W_1\bdot\ldots\bdot W_{2m-2+k_m}$ be a strong block decomposition, and if $k_n=0$, assume  either $g_0\in \supp(W_0)$ or else $g_0\in \supp(W_1)$ and $\varphi(g_0)=f_2+f_3$.
Since $\vp_{f}(\varphi(W_0))\geq 2$, there is some $h\in \supp(\wtilde W_0)$ with $\varphi(h)=f$. Since $\vp_f(\varphi(W_1))\geq 1$ with either $g_0\in \supp(W_0)$ or $\varphi(g_0)=f_2+f_3$,   there is some $g\in \supp(\wtilde W_1)$ with $\varphi(g)=f$. If $k_n=1$, then $k_\emptyset=0$, and if $k_n=0$, then $k_\emptyset\in \{0,1\}$ (as $g_0\in \supp(W_0\bdot W_1)$). Thus applying Claim D shows that $\varphi(g)=f$ is good, as claimed.

Now, since all $g\in \supp(S)$ are good, an appropriate choice of pre-images for the elements $f_1$ and $f_2$ yields  $\supp(S)\subseteq \{e_1,e_2,e_3+\alpha,e_2+e_3+\beta\}$ for some $\alpha,\beta\in \ker \varphi$, where $e_3:=xe_1+e_2$, $\varphi(e_1)=f_1$ and $\varphi(e_2)=f_2$. As before, let $S^*=W_0\bdot W_1\bdot\ldots\bdot W_{2m-2+k_m}$
be a strong block decomposition, and if $k_n=0$, assume  either $g_0\in \supp(W_0)$ or else  $g_0\in \supp(W_1)$ and $\varphi(g_0)=f_2+f_3$.
By choosing $g_0\in g_0+\ker \varphi$ appropriately, we can assume $g_0\in \{e_1,e_2,e_3+\alpha,e_2+e_3+\beta\}$ for $k_n=0$ as well.
Since $x,\,x^*\in [2,n-2]$, we have  $\vp_{e_1}(W_0)=n-1>x$ and $\vp_{e_2}(W_0)=x^*>1$, whence $e_1^{[x]}\bdot e_2\mid \wtilde W_0$ and $e_3+\alpha\mid \wtilde W_1$. Thus, setting $W'_0=W_0\bdot (e_1^{[x]}\bdot e_2)^{[-1]}\bdot (e_3+\alpha)$, $W'_1=W_1\bdot (e_3+\alpha)^{[-1]}\bdot e_1^{[x]}\bdot e_2$ and $W'_i=W_i$ for $i\geq 2$, we obtain a weak block decomposition $S^*=W'_0\bdot W'_1\bdot\ldots\bdot W'_{2m-2+k_m}$ say  with associated index $k'_\emptyset$ and associated sequence $S'_\sigma$.
Since $e_1^{[x]}\bdot e_2\mid \wtilde W_0$ and $e_3+\alpha\mid \wtilde W_1$, we have $k'_\emptyset=k_\emptyset\in \{0,1\}$ when $k_n=0$, while $|W'_0|=2n-1-x\geq n+1$ ensures that $k'_\emptyset=k_\emptyset=0$ for $k_n=1$.
As a result,
Claim B and Lemma \ref{lem-perturb-solo}  imply $S_\sigma=S'_\sigma$ with $\alpha=0$. Similarly, since $x,\,x^*\in [2,n-2]$,
 we have  $\vp_{e_1}(W_0)=n-1>n-x$ and $\vp_{e_3}(W_0)=n-x^*>1$, whence $e_1^{[n-x]}\bdot e_3\mid \wtilde W_0$ and $e_2\mid \wtilde W_1$.
Setting $W'_0=W_0\bdot (e_1^{[n-x]}\bdot e_3)^{[-1]}\bdot e_2$, $W'_1=W_1\bdot e_2^{[-1]}\bdot e_1^{[n-x]}\bdot e_3$ and $W'_i=W_i$ for $i\geq 2$, we obtain a weak block decomposition $S^*=W'_0\bdot W'_1\bdot\ldots\bdot W'_{2m-2+k_m}$ with $|W'_0|=2n-1-(n-x)\geq n+1$ having associated index $k'_\emptyset=k_\emptyset$ and associated sequence $S'_\sigma$. Thus  Claim B and Lemma \ref{lem-perturb-solo} yield $S_\sigma=S'_\sigma$ and  $ne_1=(n-x)e_1+e_3-e_2=0$. We thus obtain a zero-sum subsequence $e_1^{[n]}\mid S$, contradicting that  $0\notin\Sigma_{\leq n}(S)$ (which holds by \eqref{start-hyp}), unless $a=1$, $k_n=0$ and  $\varphi(g_0)=f_1$.
However, in such case, there are $|S^*|-n-1=(2m+k_m-1)n-2\geq 2mn-2$ terms of $S$ equal to either $e_2$ or $e_3$, so that the pigeonhole principle yields that either $e_2$ or $e_3$ has multiplicity at least $mn-1$ in $S$. If either has multiplicity at least $mn$, then $e_2^{[mn]}$ or $e_3^{[mn]}$ is a zero-sum subsequence of length $mn$, contradicting that  $0\notin \Sigma_{\leq mn}(S)$ by \eqref{start-hyp}. On the other hand, if both $e_2$ and $e_3=xe_1+e_2$ have multiplicity $mn-1\geq n$ (as $m\geq 2$), then $e_2^{[mn-n]}\bdot e_3^{[n]}$ is a zero-sum subsequence of length $mn$ (as $ne_1=0$), again contradicting that $0\notin \Sigma_{\leq mn}(S)$.
\end{proof}

In view of Claim E, we now assume Claim C.1 holds for the remainder of the proof, say with basis $(f_1,f_2)$. Then $\supp(\varphi(S^*))\subseteq f_1\cup \big(\la f_1\ra+f_2\big)$, implying  that the number of terms from $\la f_1\ra+f_2$ in any zero-sum subsequence of $\varphi(S)$ must be congruent to $0$ modulo $n$. As a result, if $S^*=W_0\bdot W_1\bdot \ldots\bdot W_{2m-2+k_m}$ is any block decomposition, then  since $|W_i|=n$ for $i\geq 1$ and $|W_i|=2n-1$, we find that \be\label{struct}\varphi(W_0)=f_1^{[n-1]}\bdot {\prod}^\bullet_{i\in [1,n]}(x_if_1+f_2)\;\und\;\varphi(W_i)\in \{f_1^{[n]}, {\prod}^\bullet_{i\in [1,n]}(c_if_1+f_2)\}\;\mbox{ for $i\geq 1$},\ee where  $x_1,\ldots,x_n,c_1,\ldots,c_n\in [0,n-1]$ with $c_1+\ldots+c_n\equiv 0 \mod n$ and  $x_1+\ldots+x_n\equiv 1\mod n$.
Note, if  $(f_1,f_2)$ is a basis for which Claim C.1 holds, then so is $(f_1,xf_1+f_2)$ for any $x\in \Z$.

\subsection*{Claim F} If $n=2$, then all terms $x\in \supp(S)$ are good.

  \begin{proof}
Let $S^*=W_0\bdot W_1\bdot\ldots\bdot W_{2m-2+k_m}$ be a block decomposition. Then $\varphi(W_0)$ is a minimal zero-sum  sequence of length $2n-1=3$ by Claim A, so $\varphi(W_0)=f_1\bdot f_2\bdot (f_1+f_2)$ with $f_1$, $f_2$ and $f_1+f_2$ the three nonzero elements of $\varphi(G)\cong C_n\oplus C_n=C_2\oplus C_2$.  By Claim C.1, we have $\supp(\varphi(S))\subseteq \{f_1,f_2,f_1+f_2\}=\varphi(G)\setminus \{0\}$, and $|\supp(\varphi(W_i))|=1$ for $i\in [1,2m-2+k_m]$, allowing us to assume $g_0\in \supp(W_0)$ when $k_n=0$, and by choosing $f_1$ and $f_2$ appropriately, we can w.l.o.g. assume $\varphi(g_0)=f_1+f_2$.
If $\vp_{f_1}(S)=1$, then the definition of good holds trivially for $f_1$. Otherwise, there is some $W_j$ with $f_1\in \supp(W_j)$ and $j\geq 1$, in which case Claim
D, implies that $f_1$ is good. Thus $f_1$ is good, and the same argument  shows that $f_2$ is good. If $k_n=1$, the argument also shows  $f_1+f_2$ is good.
The proof of the claim is now complete unless $\vp_{f_1+f_2}(\varphi(S))\geq 2$ with $k_n=0$ and $S^*=S\bdot g_0$, which we now assume. Note $\varphi(g_0)=f_1+f_2$, so $W'_1=(f_1+f_2)\bdot \varphi(g_0)$ is a length two zero-sum dividing $\varphi(S\bdot g_0)$. Applying the argument showing the existence of a block decomposition, it follows that there is a block decomposition $W'_0\bdot W'_1\bdot\ldots\bdot W'_{2m-2+k_m}$ of $S^*$ with $\varphi(W'_1)=(f_1+f_2)\bdot \varphi(g_0)$ and $k_\emptyset=1$. Since $\vp_{f_1+f_2}(\varphi(S))\geq 2$, it follows that there is some $j\in [0,2m-2+k_m]\setminus \{1\}$ with $f_1+f_2\in \supp(\varphi(W_j))$, and now Claim D, applied to the block decomposition $W'_0\bdot W'_1\bdot\ldots\bdot W'_{2m-2+k_m}$, implies that $f_1+f_2$ is good, completing the claim.
\end{proof}

\subsection*{Claim G} Suppose $S^*=W_0\bdot W_1\bdot\ldots\bdot W_{2m-2+k_m}$ is a  block decomposition  with $k_\emptyset =0$. If    $g_1\in \supp(\wtilde W_{j_1})$, \ $g_2\in \supp(\wtilde W_{j_2})$  and $\varphi(g_1)=\varphi(g_2)$, where $j_1,\,j_2\in [0,2m-2+k_m]$ are \emph{distinct}, then $g_1=g_2$ is good.

  \begin{proof} In view of Claim F, we can assume $n\geq 3$.
Since $\vp_{f_1}(\varphi(W_0))=n-1\geq 2$ and $|W_0|-\vp_{f_1}(\varphi(W_0))=n\geq 3$ by \eqref{struct}, we have and can w.l.o.g. assume (by possibly exchanging $f_2$ for an  appropriate alternative from $\la f_1\ra+f_2$) that \be\label{redblue}f_1\in \supp(\varphi(\wtilde W_0))\quad\und\quad f_2\in \supp(\varphi(\wtilde W_0)).\ee
If $j_1=0$ or $j_2=0$, then Claim D  implies $g_1=g_2$ is good (as $k_\emptyset=0$).
 Therefore we can w.l.o.g. assume $j_1=1$ and $j_2=2$.
By Claim C.1, we have $\supp(\varphi(S))\subseteq \{f_1\}\cup \big(\la f_1\ra+f_2\big).$  If $\varphi(g_1)=f_1\in \supp(\varphi(\wtilde W_0))$, then Claim D applied with $k_\emptyset=0$ and $j=1$ implies that $g_1$ is good. Therefore, in view of \eqref{struct}, we can assume \be\label{supptig}\supp(\varphi(W_1))\subseteq \la f_1\ra+f_2\quad\und\quad\supp(\varphi(W_2))\subseteq \la f_1\ra+f_2,\ee with the latter following by an analogous argument. In particular, $$\varphi(g_1)=\varphi(g_2)=x_1f_1+f_2\quad\mbox{
for some $x_1\in [0,n-1]$}.$$

Suppose $k_n=1$. Then $|W_0\bdot W_1\bdot g_1^{[-1]}|=3n-2=\eta(C_n\oplus C_n)$, ensuring by Claim A that $W_0\bdot W_1\bdot g_1^{[-1]}$ contains an $n$-term subsequence $W'_0$ with $\varphi(W'_0)$ zero-sum. Setting $W'_0=W_0\bdot W_1\bdot (W'_1)^{[-1]}$, it follows that $S^*=W'_0\bdot W'_1\bdot W_2\bdot\ldots\bdot W_{2m-2+k_m}$ is a block decomposition with $g_1\in \supp(W'_0)$, \ $g_2\in \supp(W'_2)$ and associated index $k'_\emptyset=0$, in which case Claim D implies that $g_1=g_2$ is good, as desired.

Suppose $k_n=0$. Since $k_\emptyset=0$, we have $g_0\in \supp(W_0)$ with  $|W_0\bdot W_1\bdot g_1^{[-1]}\bdot g_0^{[-1]}|=3n-3$. If $W_0\bdot W_1\bdot g_1^{[-1]}\bdot g_0^{[-1]}$ contains an $n$-term subsequence $W'_1$ with $\varphi(W'_0)$ zero-sum, then setting $W'_0=W_0\bdot W_1\bdot (W'_1)^{[-1]}$, it follows that $S^*=W'_0\bdot W'_1\bdot W_2\bdot\ldots\bdot W_{2m-2+k_m}$ is a block decomposition with $g_0\bdot g_1\mid W'_0$, \ $g_2\in \supp(W_2)$ and associated index $k'_\emptyset=0$, in which case Claim D implies that $g_1=g_2$ is good.
Therefore, in view of Claim A, we can instead assume $0\notin \Sigma_{\leq n}(\varphi(W_0\bdot W_1\bdot g_1^{[-1]}\bdot g_0^{[-1]}))$. We have $f_1,\,f_2\in \supp(\varphi(W_0\bdot g_0^{[-1]}))\subseteq \supp(W'_0)$  by \eqref{redblue}, so  applying the established Conjecture \ref{conj-shortzs}.4  to $\varphi(W_0\bdot W_1\bdot g_1^{[-1]}\bdot g_0^{[-1]})$ yields \be\label{sqware}\varphi(W_0\bdot W_1\bdot g_1^{[-1]}\bdot g_0^{[-1]})=f_1^{[n-1]}\bdot f_2^{[n-1]}\bdot f_3^{[n-1]}\quad\mbox{ for some $f_3=x_3f_1+f_2,$}\ee where
we have $f_3=x_3f_1+f_2$ since  $\supp(\varphi(S))\subseteq \{f_1\}\cup \big(\la f_1\ra+f_2\big)$. Moreover, since $(f_2,f_3)$ must be a basis, it follows that $$\gcd(x_3,n)=1.$$

By \eqref{sqware}, we have  \be\label{e123}\varphi(g_0)+\varphi(g_1)=-\sigma(\varphi(W_0\bdot W_1\bdot g_0^{[-1]}\bdot g_1^{[-1]}))=f_1+f_2+f_3=(1+x_3)f_2+2f_2.\ee %for some $x_0\in [0,n-1]$ with $x_0\equiv 1+x_3-x_1\mod n$,
%meaning \be\label{e123}\varphi(g_0)+\varphi(g_1)=f_1+f_2+f_3=(1+x_3)f_1+2f_2.\ee
Observe that $\Sigma_{n-2}(f_2^{[n-2]}\bdot f_3^{[n-2]})=\{xf_1-2f_2:\; x\in {\underbrace{\{0,x_3\}+\ldots+\{0,x_3\}}}_{n-2}\}$. Since $\gcd(x_3,n)=1$, it follows that ${\underbrace{\{0,x_3\}+\ldots+\{0,x_3\}}}_{n-2}$ contains all residue classes modulo $n$ except $(n-1)x_3\equiv -x_3\mod n$. As a result, since $-1-x_3\not\equiv -x_3\mod n$, it follows from \eqref{e123} that $-\varphi(g_0)-\varphi(g_1)\in \Sigma_{n-2}(f_2^{[n-2]}\bdot f_3^{[n-2]})$, which means (recall \eqref{sqware}) that there is an $n$-term subsequence $W'_1\mid W_0\bdot W_1$ with $g_0\bdot g_1\mid W'_1$ and $\varphi(W'_1)$ zero-sum. Letting $W_0'=W_0\bdot W_1\bdot (W'_1)^{[-1]}$, it follows that $S^*=W'_0\bdot W'_1\bdot W_2\ldots\ldots W_{2m-2+k_m}$ is a block decomposition with $g_1\in \supp(W'_1\bdot g_0^{[-1]})$ and $g_2\in \supp(W_2)$. It now follows from Claim D, applied to this block decomposition with $k_\emptyset=1$ and $j=2$, that $g_1=g_2$ is good, as desired.
\end{proof}

\smallskip

\textbf{CASE 1.} $n=2$.

\smallskip

In this case, $\supp(\varphi(S^*))\subseteq \{f_1\}\cup\big(\la f_1\ra+f_2\big)=\{f_1,f_2,f_1+f_2\}$ with $f_1$, $f_2$ and $f_1+f_2$ the three nonzero elements of $\varphi(G)=C_n\oplus C_n=C_2\oplus C_2$. Claim F ensures that all terms of $S$ are good, so (choosing $g_0\in g_0+\ker \varphi$ appropriately when $k_n=0$) we find $$\supp(S^*)=\{e_1,e_2,e_1+e_2+\alpha\}$$ for some $e_1,e_2,\alpha\in G$ with  $$me_1=\varphi(e_1)=f_1,\quad me_2=\varphi(e_2)=f_2 \quad\und\quad m\alpha=\varphi(\alpha)=0.$$ Let $S^*=W_0\bdot\ldots\bdot W_{2m-2+k_m}$ be a block decomposition with $k_\emptyset=0$ having associated sequence $S_\sigma=\prod^\bullet_{i\in [1,2m-2+k_m]}\sigma(W_i)$. In view of \eqref{struct}, we have $W_0=e_1\bdot e_2\bdot (e_1+e_2+\alpha)$ and $W_i\in \{e_1^{[2]}, e_2^{[2]},(e_1+e_2+\alpha)^{[2]}\}$ for $i\geq 1$. Thus each term in $S_\sigma$ is either equal to $2e_1$, $2e_2$ or $2e_1+2e_2+2\alpha$. In view of Claim B, we know Conjecture \ref{conj-shortzs} holds for $S_\sigma$. Since \emph{all} bases for $\varphi(G)
=C_n\oplus C_n$ satisfy Claim C.1 when $n=2$, we can replace the basis $(f_1,f_2)$ by any alternative one.
Thus we can w.l.o.g. assume Conjecture \ref{conj-shortzs} holds for $S_\sigma$ using the basis $(2e_1,2e_2)$ with $\supp(S_\sigma)=\{2e_1,2e_2,2e_1+2e_2+2\alpha\}$.

Suppose $k_m=1<m-1$. Then $m\geq 3$, $k=k_mn+k_n\in \{2,3\}$, and  Conjecture \ref{conj-shortzs} implies that $2e_1$ occurs with multiplicity $m-1$ in $S_\sigma$, \ $2e_2$ occurs with multiplicity $x\geq 1$ in $S_\sigma$, \ $2e_1+2e_2+2\alpha$ occurs with multiplicity $m-x\geq 1$ in $S_\sigma$, and $(2e_1+2e_2+2\alpha)-2e_2\in \la 2e_1\ra$, whence $e_1+e_2+\alpha=ye_1+(1+m)e_2$ or $ye_1+e_2$ for some $y\in [0,2m-1]$. We can assume the latter does not occur, else Lemma \ref{lem-genPropB} yields the desired structure for $S$. Hence,
by swapping the basis $(f_1,f_2)$ with $(f_1,f_1+f_2)$ if need be, we can assume $x\geq m-x$ and   $$S^*=e_1^{[2m-1]}\bdot e_2^{[2x+1]}\bdot (ye_1+(1+m)e_2)^{[2(m-x)+1]}$$ for some $x\in [\frac{m}{2},m-1]$ and $y\in [0,2m-1]$.
If $y=0$, then $e_2^{[m-1]}\bdot (1+m)e_2$ is a zero-sum subsequence of $S$ with length $m$, contrary to \eqref{start-hyp}. Therefore $y\geq 1$.
If $y\geq 2$, or $k_n=1$, or $k_n=0$ with $\varphi(g_0)\neq f_1$, then $e_1^{[2m-y]}\bdot e_2^{[m-1]}\bdot (ye_1+(1+m)e_2)$ is a nontrivial zero-sum subsequence of $S$ with length $3m-y\leq 3m-1\leq 4m-4\leq 2mn-1-k$, contradicting \eqref{start-hyp}. On the other hand, if $y=1$, $k_n=0$ and $\varphi(g_0)=f_1$,  then  $e_1^{[2m-3]}\bdot e_2^{[m-3]}\bdot (e_1+(1+m)e_2)^{[3]}$ is  a nontrivial zero-sum subsequence of $S$ with length $3m-3\leq 4m-4\leq 2mn-1-k$, again contradicting \eqref{start-hyp}. So we can now assume either $k_m\in [2,m-1]$ or $m=2$.

In this case, Conjecture \ref{conj-shortzs} holding for $S_\sigma$ with basis $(2e_1,2e_2)$ means  $2e_1$ and $2e_2$ occur with multiplicity $m-1$ in $S_\sigma$, \  $2e_1+2e_2+2\alpha$ occurs with multiplicity $k_m$ in $S_\sigma$, \be\label{twocase}\la 2e_1+2\alpha\ra=\la 2e_1\ra,\quad \mbox{ and either } \quad 2\alpha=0 \;\mbox{ or }\; k_m=m-1.\ee Moreover, both $2\alpha=0$ and $k_m=m-1=1$ when $m=2$. It follows that $$S^*=e_1^{[2m-1]}\bdot e_2^{[2m-1]}\bdot (e_1+e_2+\alpha)^{[2k_m+1]}.$$

If $H=\la e_1,e_2\ra$ is a proper subgroup, then  $S$ contains a subsequence with two distinct terms from $H$ and length at least $4m-3\geq \eta(H)-1$, and thus contains a nontrivial zero-sum of length at most $\exp(H)\leq 2m$ using the established Conjecture \ref{conj-shortzs}.4, contrary to \eqref{start-hyp}. Therefore $H=\la e_1,e_2\ra=G$, forcing $(e_1,e_2)$ to be a basis for $G=C_{2m}\oplus C_{2m}$. If $\alpha\in \la e_1\ra$ or $\alpha\in \la e_2\ra$, then Lemma \ref{lem-genPropB} implies that $S$ has the desired structure, completing the proof. Hence we may assume otherwise, so  in view of
$m\alpha=0$ and $\la 2e_1+2\alpha\ra=\la 2e_1\ra$, it follows that $$\alpha=xe_1+me_2\quad\mbox{ for some $x\in [1,2m-1]$}$$ with $m$ even and $$\ord( 2(1+x)e_1)=\ord(2e_1+2\alpha)=\ord(2e_1)=m.$$
%whence $\gcd(1+x,m)=1$, forcing $x\in [2,2m-2]$ to be even as well.
We have two final subcases based upon which possibility occurs in \eqref{twocase}.

Suppose $k_m=m-1\geq 2$. Then $e_1+e_2+\alpha=(1+x)e_1+(1+m)e_2$. For  $r\in [0,\frac{m}{2}-1]$, we have $T_r:=e_2^{[m-2r-1]}\bdot ((1+x)e_1+(1+m)e_2)^{[2r+1]}$ as a subsequence of $S$ with sum $(1+x)e_1+r\cdot 2(1+x)e_1$. Since $\ord(2(1+x)e_1)=m$, it follows that  $\{\sigma(T_r): r\in [0,\frac{m}{2}-1]\}\subseteq (1+x)e_1+[1,m]_{2e_1}$ is a subset of cardinality $\frac{m}{2}$, so there must be some $r\in [0,\frac{m}{2}-1]$ such that $\sigma(T_r)=ye_1$ for some  $y\in [1,2m]$ with $y\geq 2(\frac{m}{2}-1)+1=m-1\geq 2$. It follows that $e_1^{[2m-y]}\bdot T_r$ is a nontrivial zero-sum subsequence of $S$ with length $3m-y\leq 2m+1$, contradicting \eqref{start-hyp} (as $k\leq 2m-2$). %So we may now assume $k_m\in [1,m-2]$, implying $m\geq 3$, and that $2\alpha=0$.

Suppose  $2\alpha=0$.
Then  $e_1+e_2+\alpha=(1+m)e_1+(1+m)e_2$, else Lemma \ref{lem-genPropB} yields the desired structure for $S$.
%Since $k_m\in [1,m-2]$, we have $m\geq 3$.
If $2k_m-1\leq m$, then $e_1^{[m-2k_m+1]}\bdot e_2^{[m-2k_m+1]}\bdot ((1+m)e_1+(1+m)e_2)^{2k_m-1}$ is a nontrivial zero-sum subsequence of $S$ with length $2m-2k_m+1\leq 2m-1$, contradicting \eqref{start-hyp}. On the other hand, if $2k_m-1\geq m+1$, then
$e_1\bdot e_2\bdot ((1+m)e_1+(1+m)e_2)^{m-1}$ is a nontrivial zero-sum subsequence of $S$ with length $m+1\leq 2m+1$, again contradicting \eqref{start-hyp} and completing the case.

\smallskip

\textbf{CASE 2.} $n\geq 3$.

\smallskip

Since $n\geq 3$, if $S^*=W_0\bdot W_1\bdot\ldots\bdot W_{2m-2+k_m}$ is any block decomposition, then \eqref{struct} ensures
\be\label{onceblock}f_1\in \supp(\varphi(\wtilde W_0)).\ee

\smallskip

\textbf{Claim H.} The term $f_1\in \supp(\varphi(S))$ is good.

\begin{proof}
Let $S^*=W_0\bdot W_1\bdot \ldots\bdot W_{2m-2+k_m}$ be a block decomposition  with associated index $k_\emptyset=0$.
%If $n=2$, then the  claim follows from Claim F, so we can assume $n\geq 3$.
If $f_1\in \supp(\varphi(S\bdot W_0^{[-1]}))$, then \eqref{onceblock} together with Claim D implies that $f_1$ is good. Therefore we can instead assume  \be\label{chole}\supp(\varphi(S\bdot W_0^{[-1]}))\subseteq \la f_1\ra+f_2.\ee We can also assume \be\label{vpup}\vp_{f_1}(\varphi(S))\geq 2,\ee lest $f_1$ being good holds trivially.

Let us show there are $g\in \supp(\wtilde W_0)$ and $h\in \supp(S\bdot W_0^{[-1]})$ with $\varphi(g),\,\varphi(h)\in \la f_1\ra+f_2$ and $$\varphi(g)-\varphi(h)=zf_1\quad \mbox{for some } z\in [2,n-1].$$
Assuming this fails, let $xf_1+f_2\in \supp(\varphi(\wtilde W_0))$. Then \eqref{chole} implies $\supp(\varphi(S\bdot W_0^{[-1]}))\subseteq \{xf_1+f_2, (x-1)f_1+f_2\}$. If $\vp_{xf_1+f_2}(\varphi(S))\geq mn>n$, then the term $g$ is good by Claim G, in turn implying that $\vp_g(S)\geq mn$, where $g\in \supp(S)$ is the unique term with $\varphi(g)=xf_1+f_2$, whence $0\in \Sigma_{mn}(S)$, contrary to \eqref{start-hyp}. Therefore we can assume $\vp_{xf_1+f_2}(\varphi(S))\leq mn-1$, and thus $\vp_{(x-1)f_1+f_2}(\varphi(S\bdot W_0^{[-1]}))\geq (2m-2+k_m)n-mn+1\geq mn-n+1>n$.
Claim G now ensures that the term $(x-1)f_1+f_2$ is also good, and thus has multiplicity at most $mn-1$ in $\varphi(S)$ lest we obtain the same contradiction as before. There are at least $(2m-2+k_m)n+n-1\geq 2mn-1$ terms of $\varphi(S)$ from  $\la f_1\ra+f_2$. As a result,  it follows that there is some $x'f_1+f_2\in\supp(\varphi(\wtilde W_0))$ with $x'f_1\notin \{ xf_1, (x-1)f_1\}$. But now, taking $g'=x'f_1+f_2$ and $h'=(x-1)f_1+f_2\in \supp(\varphi(S\bdot W_0^{[-1]}))$, we find that $g'-h'=zf_1$ with $z\in [2,n-1]$, as desired. Thus the existence of $g$ and $h$ is established.

Let $j\in [1,2m-2+k_m]$ be an index with $h\in \supp(W_j)$. In view of \eqref{chole} and \eqref{vpup}, let $g_1\bdot g_2\mid \wtilde W_0$ be a length two subsequence with $\varphi(g_1)=\varphi(g_2)=f_1$. Since
 $1\leq n-z\leq n-2$ and $\vp_{f_1}(\varphi(W_0\bdot g_2^{[-1]}))=n-2$, it follows that there is a subsequence  $T\mid W_0\bdot g_2^{[-1]}$ with $\varphi(T)=f_1^{[n-z]}$ and $g_1\in \supp(T)$. Since $\varphi(g)\in \la f_1\ra+f_2$, we have
$g\notin \supp(T)$.
Set  $$W'_0=W_0\bdot T^{[-1]}\bdot g^{[-1]}\bdot h\quad\und\quad W'_j=W_j\bdot h^{[-1]}\bdot g\bdot T$$ with $W'_i=W_i$ for $i\neq 0,j$. Then, by construction,
$S^*=W'_0\bdot W'_1\bdot\ldots\bdot W'_{2m-2+k_m}$ is a weak block decomposition with associated index $k'_\emptyset\in \{0,j\}$. Moreover,  $g_2\in \supp(\wtilde W'_0)$  and $g_1\in \supp(\wtilde W'_j)$ with $\varphi(g_1)=\varphi(g_2)=f_1$. Thus applying Claim D to $S^*=W'_0\bdot W'_1\bdot\ldots\bdot W'_{2m-2+k_m}$  implies $f_1$ is good, completing the claim.
\end{proof}

Let $e_1\in \supp(S)$ with $\varphi(e_1)=f_1$, which exists in view of \eqref{onceblock}. By Claim H, every $g\in \supp(S)$ with $\varphi(g)=f_1$ has $g=e_1$, and if $k_n=0$ with $\varphi(g_0)=f_1=\varphi(e_1)$, we can choose  $g_0\in g_0+\ker \varphi$ appropriately so that $g_0=e_1$, thereby ensuring that every  $g\in \supp(S^*)$ with $\varphi(g)=f_1$ has $g=e_1$.

\subsection*{Claim I} If $g,\,h\in \supp(S)$ with $\varphi(g),\varphi(h)\in \la f_1\ra+f_2$, then $g-h\in \la e_1\ra$.

\begin{proof}
Let $S^*=W_0\bdot W_1\bdot\ldots\bdot W_{2m-2+k_m}$ be a block decomposition with  associated index $k_\emptyset=0$ and  associated sequence  $S_\sigma=\prod^\bullet_{i\in [1,2m-2+k_m]}\sigma(W_i)$. Since $f_1$ is good (by Claim H), we have $\vp_{f_1}(\varphi(S))\leq mn-1$, lest $S$ contain an $mn$-term zero-sum, contrary to \eqref{start-hyp}.
Thus, since each $\varphi(W_i)$, for $i\in [1,2m-2+k_m]$, either consists of $n$ terms equal to $f_1$ or no terms equal to $f_1$ (in view of \eqref{struct}), it follows that $\vp_{f_1}(\varphi(S\bdot W_0^{[-1]}))\leq (m-1)n$, meaning there are at least $(2m-2+k_m-(m-1))n\geq mn$ terms of $\varphi(S\bdot W_0^{[-1]})$ from $\la f_1\ra+f_2$. These terms cannot all be equal to each other, lest they would be good by Claim G giving rise to an element with multiplicity at least $mn$ in $S$, contradicting \eqref{start-hyp} as before. Therefore \be\label{supp2}|\supp(\varphi(S\bdot W_0^{[-1]}))\setminus \{f_1\}|\geq 2.\ee

For $x\in [0,n-1]$,  let $L_x\mid \wtilde W_0$ be the subsequence of $\wtilde W_0$ consisting of all terms $g$ with $\varphi(g)=xf_1+f_2$, and let $R_x\mid S\bdot W_0^{[-1]}$ be the subsequence of $S\bdot W_0^{[-1]}$ consisting of all terms $g$ with $\varphi(g)=xf_1+f_2$.
Let $I_L\subseteq [0,n-1]$ be all those $x\in [0,n-1]$ with $L_x$ nontrivial, and let $I_R\subseteq [0,n-1]$ be all those $x\in [0,n-1]$ with $R_x$ nontrivial.  By a slight abuse of notation, we consider the subscripts on the $L_x$ and $R_x$ modulo $n$.
In view of \eqref{supp2}, $$|I_R|\geq 2.$$

Let $g\in \supp(\wtilde W_0)$ and $h\in \supp(S\bdot W_0^{[-1]})$ be arbitrary with $\varphi(g),\,\varphi(h)\in \la f_1\ra+f_2$, and  let $$\varphi(g)-\varphi(h)=zf_1\quad\mbox{ with $z\in [1,n]$}.$$

Suppose $k_n=0$ and $\varphi(g_0)\neq f_1$.
Then $e_1^{[n-z]}\bdot g\mid \wtilde W_0$. Set $W'_0=W_0\bdot \big(e_1^{[n-z]}\bdot g\big)^{[-1]}\bdot h$, $W'_j=W_j\bdot h^{[-1]}\bdot e_1^{[n-z]}\bdot g$ and $W'_i=W_i$ for $i\neq 0,j$, where $h\in \supp(W_j)$.
Then $S^*=W'_0\bdot W'_1\bdot\ldots\ldots\bdot W_{2m-2+k_m}$ is a weak block decomposition with  $g_0\in \supp(W'_0)$, associated index $k'_\emptyset =k_\emptyset =0$ (as $g\in \supp(W'_0)$) and associated sequence  $S'_\sigma=\prod^\bullet_{i\in [1,2m-2+k_m]}\sigma(W'_i)$. Note that $S'_\sigma$ is obtained from $S_\sigma$ by replacing the term $\sigma(W_j)$ by  $\sigma(W'_j)=\sigma(W_j)-h+g+(n-z)e_1$.
 By Lemma \ref{lem-perturb-solo} and Claim B, we must have $S_\sigma=S'_\sigma$, implying $\sigma(W_j)=\sigma(W'_j)$ and $g-h\in \la e_1\ra$.
As this is true for arbitrary $g\in \supp(\wtilde W_0)$ and $h\in \supp(S\bdot W_0^{[-1]})$, the claim is complete in this case, allowing us to assume $k_n=1$ or $\varphi(g_0)=f_1.$
In particular, it now follows from \eqref{struct} that  $$|I_L|\geq 2$$  (as $x_1+\ldots+x_n\equiv 1\mod n$ ensures not all $x_i$ are equal to each other).

Suppose $z\geq 2$ for $g$ and $h$ as before. Then   $e_1^{[n-z]}\bdot g\mid \wtilde W_0$. Set $W'_0=W_0\bdot \big(e_1^{[n-z]}\bdot g\big)^{[-1]}\bdot h$, \ $W'_j=W_j\bdot h^{[-1]}\bdot e_1^{[n-z]}\bdot g$ and $W'_i=W_i$ for $i\neq 0,j$, where $h\in \supp(W_j)$. Then $S^*=W'_0\bdot W'_1\bdot\ldots\ldots\bdot W_{2m-2+k_m}$ is a weak block decomposition with associated index $k'_\emptyset$. If $k_n=0$, then  $g_0\in \supp(W'_0)$, ensuring $k'_\emptyset=k_\emptyset=0$; if $k_n=1$, then $|W'_0|=|W_0|-(n-z)\geq n+1$ follows in view of $z\geq 2$, also ensuring that  $k'_\emptyset=k_\emptyset=0$. Applying Lemma \ref{lem-perturb-solo} and Claim B, we find that $g$ and $h$ are from the same $\la e_1\ra$-coset as before.

The argument from the previous paragraph shows that, for any $x\in I_L$, all  terms from  $L_x$ are from the same $\la e_1\ra$-coset as all terms from $R_y$, for any $y\not\equiv x-1\mod n$.
If all terms from $\prod^\bullet_{y\in I_R} R_y$ are from the same
$\la e_1\ra$-coset, then each $x\in I_L$ would have all terms from $L_x$ being from the same $\la e_1\ra$-coset as all terms from some $R_y$ with  $y\in I_R$ (as $|I_R|\geq 2$), and thus from the same $\la e_1\ra$-coset that contains all terms from $\prod^\bullet_{y\in I_R}R_y$.
As this would be true for any $x\in I_L$, there would only be one $\la e_1\ra$-coset containing all $g\in \supp(S)$ with $\varphi(g)\in \la f_1\ra+f_2$, completing the proof of the claim. Therefore we can instead assume we need at least two $\la e_1\ra$-cosets to cover all terms from $\prod^\bullet_{y\in I_R}R_y$. In particular, for any $x\in I_L$, we must have $x-1\in I_R$, so $I_L-1\subseteq I_R\mod n$.
 Likewise, since $|I_L|\geq 2$, we can assume  we need at least two $\la e_1\ra$-cosets to cover all terms from $\prod^\bullet_{x\in I_L}L_x$, and thus for any $y\in I_R$, we have $y+1\in I_L$, so $I_R+1\subseteq I_L\mod n$. It follows that $|I_L|=|I_R|$ with $$I_R=\{x-1:\;x\in I_L\}\mod n.$$

Suppose $|I_R|\geq 3$.
 Letting  $x_1,\,x_2\in I_L$ be distinct, then all terms from $\prod^\bullet_{y\in I_R\setminus \{x_1-1\}}R_y$ are from the same $\la e_1\ra$-cost as the terms from $L_{x_1}$, while  all terms from $\prod^\bullet_{y\in I_R\setminus \{x_2-1\}}R_y$ are from the same $\la e_1\ra$-cost as the terms from $L_{x_2}$. Since $|I_R|\geq 3$, there would be a common element $y\in I_R\setminus \{x_1-1,x_2-1\}$, forcing all terms from  $\prod^\bullet_{y\in I_R}R_y$ to be from the same $\la e_1\ra$-coset, which we just assumed was not the case. So we instead conclude that $|I_L|=|I_R|=2$. Let $$I_L=\{x,y\}\quad\und\quad I_R=\{x-1,y-1\}\mod n.$$

In view of \eqref{struct}, any $W_j$ with $j\in [1,2m-2+k_m]$ that contains a term $h$ with $\varphi(h)\in \la f_1\ra+f_2$ must have all its terms from $\la f_1\ra+f_2$. Thus, since $|W_j|=n\geq 3$ and $|I_R|=2$, the Pigeonhole Principle ensures that there are $h_1\bdot h_2\mid W_j$ with $\varphi(h_1)=\varphi(h_2)$, say w.l.o.g. $\varphi(h_1)=\varphi(h_2)=(x-1)f_1+f_2$. By definition of $I_L=\{x,y\}$, there are $g_1\bdot g_2\mid \wtilde W_0$ with $\varphi(g_1)=xf_1+f_2$ and $\varphi(g_2)=yf_1+f_2$. Let \be \label{zdef}z'\equiv x+y-2(x-1)\mod n\quad\mbox{ with $z'\in [1,n]$}.\ee

Suppose $z'\geq 2$. Then $e_1^{[n-z']}\bdot g_1\bdot g_2\mid \wtilde W_0$. Set $W'_0=W_0\bdot (e_1^{[n-z']}\bdot g_1\bdot g_2)^{[-1]}\bdot h_1\bdot h_2$, \ $W'_j=W_j\bdot h_1^{[-1]}\bdot h_2^{[-1]}\bdot e_1^{[n-z']}\bdot g_1\bdot g_2$ and $W'_i=W_i$ for $i\neq 0,j$.
Then $S^*=W'_0\bdot W'_1\bdot\ldots\ldots\bdot W_{2m-2+k_m}$ is a weak block decomposition say with associated index $k'_\emptyset$ and associated sequence $S'_\sigma$.
If $k_n=0$, then  $g_0\in \supp(W'_0)$, ensuring $k'_\emptyset=k_\emptyset=0$; if $k_n=1$, then $|W'_0|=|W_0|-(n-z')\geq n+1$ follows in view of $z'\geq 2$, also ensuring that  $k'_\emptyset=k_\emptyset=0$. By Lemma \ref{lem-perturb-solo} and Claim B, it follows that $S_\sigma=S'_\sigma$, and thus \be\label{dicet}\sigma(W_j)=\sigma(W'_j)=
\sigma(W_j)+(n-z')e_1+g_1+g_2-h_1-h_2.\ee
Since $\varphi(g_2)=yf_1+f_2$ and $\varphi(h_2)=(x-1)f_2+f_2$, we have $g_2\in \supp(L_y)$ and $h_2\in \supp(R_{x-1})$, so our previous argument ensures $g_2$ and $h_2$ are from the same $\la e_1\ra$-coset, and then  \eqref{dicet} implies that $g_1$ and $h_1$ are from the same $\la e_1\ra$-coset. Since $\varphi(g_1)=xf_1+f_2$ and $\varphi(h_1)=(x-1)f_2+f_2$, so $g_1\in \supp(L_x)$ and $h_1\in \supp(R_{x-1})$,  the terms from $L_x$ are from the same $\la e_1\ra$-coset as both $R_{x-1}$ and $R_{y-1}$, contradicting our assumption that we need at least two $\la e_1\ra$-cosets to cover the terms from $\prod^\bullet_{z\in I_R}R_z=R_{x-1}\bdot R_{y-1}$. So we are left to conclude $z'=1$, which by \eqref{zdef} means   \be\label{conghelp}y\equiv x-1\mod n.\ee

Since $I_R=\{x-1,y-1\}$, there is some $j\in [1,2m-2+k_m]$ and $h\in \supp(W_j)$ with $\varphi(h)=(y-1)f_1+f_2$. In view of \eqref{struct}, we have $\supp(\varphi(W_j))\subseteq \la f_1\ra+f_2$, and thus all terms from $\varphi(W_j)$ are either equal to $(x-1)f_1+f_2$ or $(y-1)f_1+f_2$. Since $\varphi(W_j)$ is an $n$-term zero-sum, it cannot have a term with multiplicity exactly $n-1$, so there must be $h_1\bdot h_2\mid W_j$ with $\varphi(h_1)=\varphi(h_2)=(y-1)f_1+f_2$. Repeating the argument of the previous paragraph swapping the roles of $y$ and  $x$, we conclude that $x\equiv y-1\mod n$. Combined with \eqref{conghelp}, it follows that $0\equiv 2\mod n$, contradicting that $n\geq 3$, which concludes the claim.
\end{proof}

Note that any $g\in \supp(S)$ with $\varphi(g)\neq f_1$ has $\varphi(g)=xf_1+f_2$ for some $x\in [0,n-1]$ by Claim C.1.
%Replacing the basis $(f_1,f_2)$ with one of the form $(f_1, xf_1+f_2)$, we can assume $\varphi(e_2)=f_2=\varphi(g)$ for some $g\in \supp(S)$. But now,
Thus, in view of Claims H and I, we see that $\supp(S)\subseteq\{e_1\}\cup \big(\la e_1\ra+e_2\big)$ for some $e_2\in \supp(S)$. This allow us to apply Lemma \ref{lem-genPropB} to $S$ to complete the proof.
\end{proof}

Next, we consider the case when $k_n\in[2,n-1]$.

\begin{proposition}\label{prop-kn-2n-1}
Let $m,\,n\geq 2$  and let $k\in [0,mn-1]$ with $n=k_mn+k_n$,  where $k_m\in [0,m-1]$ and $k_n\in [2, n-1]$. Suppose Conjecture \ref{conj-shortzs} holds for $k_n$ in $C_n\oplus C_n$. Suppose either Conjecture \ref{conj-shortzs} also holds for $k_m$ in $C_m\oplus C_m$, or else $k_m\in [1,m-2]$ and Conjecture \ref{conj-shortzs} also holds for $k_m+1$ in $C_m\oplus C_m$. Then Conjecture \ref{conj-shortzs} holds for $k$ in $C_{mn}\oplus C_{mn}$.
\end{proposition}

\begin{proof}
Let $G=C_{mn}\oplus C_{mn}$ and let $S\in \Fc(G)$ be a sequence with \be\label{conto} |S|=2mn-2+k\quad\und\quad 0\notin \Sigma_{\leq 2mn-1-k}(G).\ee
 Since $k_m\in [0,m-1]$ and $k_n\in [2,n-1]$, we have $n\geq 3$ and $k=k_mn+k_n\in [2,mn-1]$.
 Since Conjecture \ref{conj-shortzs} is known for  $k=mn-1$ (as remarked in the introduction), we can assume $k=k_mn+k_n\in [2,mn-2]$.
 We need to show Conjecture \ref{conj-shortzs}.3 holds for $S$. Let $\varphi:G\rightarrow G$ be the multiplication by $m$ homomorphism, so $\varphi(x)=mx$. Note $$\varphi(G)=mG\cong C_n\oplus C_n\quad\und\quad \ker \varphi=nG\cong C_m\oplus C_m.$$
 Define a \emph{block decomposition} of $S$ to be a factorization $$S=W\bdot W_1\bdot\ldots\bdot W_{2m-2+k_m}$$ with $1\leq |W_i|\leq n$ and $\varphi(W_i)$ zero-sum for each $i\in [1,2m-2+k_m]$. Since $\mathsf s_{\leq n}(\varphi(G))=\mathsf s_{\leq n}(C_n\oplus
 C_n)=3n-2$ and $|S|=(2m-3+k_m)n+3n-2+k_n\geq (2m-3+k_m)n+3n-2$, it follows by repeated application of the definition of $\mathsf s_{\leq n}(\varphi(G))$ that $S$ has a block decomposition.

 \subsection*{Claim A} If $S =W\bdot W_1\bdot\ldots\bdot W_{2m-2+k_m}$ is a block decomposition of $S$, then $|W_i|=n$ for all $i\in [1,2m-2+k_m]$,  $|W|=2n-2+k_n$, $0\notin \Sigma_{\leq 2n-1-k_n}(\varphi(W))$,  and  $0\notin \Sigma_{\leq n-1}(\varphi(S))$. In particular, Conjecture \ref{conj-shortzs} holds for $\varphi(W)$.

\begin{proof}
Suppose $0\in \Sigma_{\leq 2n-1-k_n}(\varphi(W))$. Then there is a nontrivial subsequence $W_0\mid W$ with $|W_0|\leq 2n-1-k_n$ and $\varphi(W_0)$ zero-sum. Now $\sigma(W_0)\bdot \sigma(W_1)\bdot\ldots\bdot \sigma(W_{2m-2+k_m})$ is a sequence of $2m-1+k_m$ terms from $\ker\varphi\cong C_m\oplus C_m$. Since $\mathsf s_{\leq 2m-1-k_m}(C_m\oplus C_m)=2m-1+k_m$, it follows that it has a nontrivial zero-sum sequence, say $\prod_{i\in I}^\bullet \sigma(W_i)$ for some nonempty $I\subseteq [0,2m-2+k_m]$ with $|I|\leq 2m-1-k_m$. But then $\prod^\bullet_{i\in I}W_i$ is a nontrivial zero-sum subsequence of $S$ with \begin{align*}|{\prod}^\bullet_{i\in I}W_i|&\leq
\max\{|W_0|,n\}+(|I|-1)n\leq 2n-1-k_n+(2m-2-k_m)n=2mn-1-k,\end{align*} contradicting  \eqref{conto}.
So we instead conclude that $0\notin \Sigma_{\leq 2n-1-k_n}(\varphi(W))$.

As a result, since $\mathsf s_{\leq 2n-1-k_n}(\varphi(G))=\mathsf s_{\leq 2n-1-k_n}(C_n\oplus C_n)=2n-1+k_n$, and since $|W_i|\leq n$ for all $i\in [1,2m-2+k_m]$, it follows that $$2n-2+k_n=2mn-2+k-(2m-2+k_m)n\leq |S|-\Sum{i=1}{2m-2+k_m}|W_i|=|W|\leq 2n-2+k_n,$$ forcing equality to hold in our estimates, i.e., $|W_i|=n$ for all $i\in [1,2m-2+k_m]$ and $|W|=2n-2+k_n$.
If $0\in \Sigma_{\leq n-1}(\varphi(S))$, then we can find a nontrivial subsequence $W'_1\mid S$ with $\varphi(W'_1)$ zero-sum and $|W'_1|\leq n-1$. Applying the argument used to show the existence of a  block decomposition, we obtain a block decomposition $S=W'\bdot W'_1\bdot \ldots\bdot W'_{2m-2+k_m}$ with $|W'_1|\leq n-1$, contradicting what was just shown. Therefore $0\notin \Sigma_{\leq n-1}(\varphi(S))$. Finally, since $|W|=2n-2+k_n$ and $0\notin \Sigma_{\leq 2n-1-k_n}(\varphi(W))$ with Conjecture \ref{conj-shortzs} holding for $k_n$ in $C_n\oplus C_n$ by hypothesis, it follows that Conjecture \ref{conj-shortzs} holds for $\varphi(W)$, completing the claim.
\end{proof}

Suppose $$S=W\bdot W_1\bdot\ldots\bdot W_{2m-2+k_m}$$ with each $\varphi(W_i)$ a nontrivial zero-sum for $i\in [1,2m-2+k_m]$ and $|W|\geq n-1+2k_n$. We call this a \emph{weak block decomposition} of $S$ with associated sequence $$S_\sigma=\sigma(W_1)\bdot\ldots\bdot \sigma(W_{2m-2+k_m})\in \Fc(\ker \varphi).$$ Since $2n-2+k_n\geq n-1+2k_n$, any block decomposition is also a weak block decomposition. Suppose  $$S=W\bdot W_0\bdot W_1\bdot\ldots\bdot W_{2m-2+k_m}$$ with each $\varphi(W_i)$ a nontrivial zero-sum for $i\in [0,2m-2+k_m]$, $|W_i|=n$ for all $i\in [1,2m-2+k_m]$, and $|W_0|\leq 3n-1-k_n$. We call this an \emph{augmented block decomposition} of $S$. In such case, $S=(W\bdot W_0)\bdot W_1\bdot\ldots\bdot W_{2m-2+k_m}$ is a block decomposition of $S$ with associated sequence $S_\sigma$, and we call $$\wtilde S_\sigma=\sigma(W_0)\bdot S_\sigma=\sigma(W_0)\bdot \sigma(W_1)\bdot\ldots\bdot \sigma(W_{2m-2+k_m})\in \Fc(\ker \varphi)$$ the associated sequence for the augmented block decomposition.
Conversely, if $S=W\bdot W_1\bdot\ldots\bdot W_{2m-2+k_m}$ is a block decomposition, then Claim A ensures that  $|W|=2n-1+(k_n-1)=\mathsf s_{\leq 2n-k_n}(\ker \varphi)$ (in view of $k_n\geq 1$). As a result, there is a nontrivial subsequence $W_0\mid W$ with $\varphi(W_0)$ zero-sum and $|W_0|\leq 2n-k_n\leq 3n-1-k_n$, ensuring  $(W\bdot W_0^{[-1]})\bdot W_0\bdot\ldots\bdot W_{2m-2+k_m}$ is  an augmented block decomposition, showing such decompositions exist.

\subsection*{Claim B}
\begin{itemize}
\item[1.]If $S=W\bdot W_1\bdot \ldots\bdot W_{2m-2+k_m}$ is a weak block decomposition with associated sequence $S_\sigma$, then $|S_\sigma|=2m-2+k_m$ and $0\notin \Sigma_{\leq 2m-1-k_m}(S_\sigma)$. If it is also a block decomposition, then Conjecture \ref{conj-shortzs} holds for $S_\sigma$.
\item[2.]
If $S=W\bdot W_0\bdot \ldots\bdot W_{2m-2+k_m}$ is an augmented block decomposition with associated sequence $\wtilde S_\sigma$, then $|\wtilde S_\sigma|=2m-1+k_m$ and $0\notin \Sigma_{\leq 2m-2-k_m}(\wtilde S_\sigma)$. Moreover, if $k_m\in [0,m-2]$, then Conjecture \ref{conj-shortzs} holds for $\wtilde S_\sigma$.
\end{itemize}

\begin{proof}
If $S=W\bdot W_1\bdot \ldots\bdot W_{2m-2+k_m}$ is a weak block decomposition with associated sequence $S_\sigma$, then $|S_\sigma|=2m-2+k_m$ holds by definition. If $0\in \Sigma_{\leq 2m-1-k_m}(S_\sigma)$, then there is some $I\subseteq [1,2m-2+k_m]$ with $1\leq |I|\leq 2m-1+k_m$ and $\prod^\bullet_{i\in I}\sigma(W_i)$ zero-sum. In such case, since $|W_i|\geq n$ for all $i\in [1,2m-2+k_m]$ by Claim A, we find that $\prod^\bullet_{i\in I} W_i$ is a nontrivial zero-sum subsequence of $S$ with length at most \begin{align*}&(|S|-|W|)-(2m-2+k_m-|I|)n=2n-2+k_n-|W|+|I|n\\&\leq 2mn-2+n+k_n-k_mn-|W|\leq 2mn-1-k_mn-k_n=2mn-1-k,\end{align*} with the final inequality in view of the definition of a weak block decomposition, which contradicts the hypothesis $0\notin \Sigma_{\leq 2mn-1-k}(S)$. Therefore $0\notin \Sigma_{\leq 2m-1-k_m}(S_\sigma)$.

If $S=W\bdot W_0\bdot \ldots\bdot W_{2m-2+k_m}$ is an augmented block decomposition with associated sequence $\wtilde S_\sigma=\sigma(W_0)\bdot S_\sigma$, then $|\wtilde S_\sigma|=2m-1+k_m$ holds by definition.
If $0\in \Sigma_{\leq 2m-2-k_m}(\wtilde S_\sigma)$, then there is some $I\subseteq [0,2m-2+k_m]$ with $1\leq |I|\leq 2m-2-k_m$ and $\prod^\bullet_{i\in I}\sigma(W_i)$ zero-sum.
In such case, since $|W_i|= n$ for all $i\geq 1$ and $|W_0|\leq 3n-1-k_n$,  we find that $\prod^\bullet_{i\in I} W_i$ is a nontrivial zero-sum subsequence of $S$ with length at most $\max\{|I|n, |W_0|+(|I|-1)n\}\leq 3n-1-k_n+(2m-3-k_m)n= 2mn-1-k$, which contradicts the hypothesis $0\notin \Sigma_{\leq 2mn-1-k}(S)$. Therefore $0\notin \Sigma_{\leq 2m-2-k_m}(\wtilde S_\sigma)$.

Suppose $S=W\bdot W_1\bdot \ldots\bdot W_{2m-2+k_m}$ is a block decomposition with associated sequence $S_\sigma$. As noted above Claim B, there exists some $W_0\mid W$ such that $(W\bdot W_0^{[-1]})\bdot W_0\bdot\ldots\bdot W_{2m-2+k_m}$ is an augmented block decomposition with associated sequence $\sigma(W_0)\bdot S_\sigma$.
As already shown, $0\notin \Sigma_{\leq 2m-2-k_m}(\sigma(W_0)\bdot S_\sigma)$ with $|\sigma(W_0)\bdot S_\sigma|=2m-1+k_m$.
By hypothesis, either Conjecture \ref{conj-shortzs} holds for $k_m$ in $C_m\oplus C_m$, or else $k_m\in [1,m-2]$ and Conjecture \ref{conj-shortzs} holds for $k_m+1$ in $C_m\oplus C_m$. In the former case, Conjecture \ref{conj-shortzs} holds for $S_\sigma$ in view of the already established $|S_\sigma|=2m-2+k_m$ and $0\notin \Sigma_{\leq 2m-1-k_m}(S_\sigma)$.
In the latter case, Conjecture \ref{conj-shortzs} holds for $\sigma(W_0)\bdot S_\sigma$ in view of the already established $|\sigma(W_0)\bdot S_\sigma|=2m-1+k_m$ and $0\notin \Sigma_{\leq 2m-2-k_m}(\Sigma(W_0)\bdot S_\sigma)$, which combined with Lemma \ref{lem-k+1-to-k} ensures that it does so with respect to some basis $(f_1,f_2)$ with $\sigma(W_0)=f_1+f_2$ and Conjecture \ref{conj-shortzs} holding for $S_\sigma$. Thus Conjecture \ref{conj-shortzs} holds for $S_\sigma$ in both cases.

Suppose $S=W\bdot W_0\bdot \ldots\bdot W_{2m-2+k_m}$ is an augmented block decomposition with associated sequence $\wtilde S_\sigma=\sigma(W_0)\bdot S_\sigma$. As noted above Claim B, $(W\bdot W_0)\bdot W_1\bdot\ldots\bdot W_{2m-2+k_m}$ is a block decomposition of $S$ with associated sequence $S_\sigma$. As already shown,
$|S_\sigma|=2m-2+k_m$ and $0\notin \Sigma_{\leq 2m-1-k_m}(S_\sigma)$ with Conjecture \ref{conj-shortzs} holding for $S_\sigma$. As a result, if $k_m\in [1,m-2]$, then Lemma \ref{lem-k-to-k+1} implies that Conjecture \ref{conj-shortzs} holds for $\wtilde S_\sigma$. On the other hand, If $k_m=0$, then Conjecture \ref{conj-shortzs} is known to hold without condition for $k_m$ and $k_m+1$ in $C_m\oplus C_m$, ensuring that Conjecture \ref{conj-shortzs} holds for $\wtilde S_\sigma$. Thus Conjecture \ref{conj-shortzs} holds for $\wtilde S_\sigma$ in both cases, completing the claim.
\end{proof}

\subsection*{Claim C} There exists a basis $(\overline e_1,\overline e_2)$ for $\varphi(G)$ such that $\supp(\varphi(S))=\{\overline e_1,\overline e_2,x\overline e_1+\overline e_2\}$   for some $x\in [1,n-1]$ with $\gcd(x,n)=1$ and either $x=1$ or $k_n=n-1$. In particular, any block decomposition $S=W\bdot W_1\bdot\ldots\bdot W_{2m-2+k_m}$ has $\varphi(W)=\overline e_1^{[n-1]}\bdot \overline e_2^{[n-1]}\bdot (x\overline e_1+\overline e_2)^{[k_n]}$ with $\varphi(W_i)\in \{\overline e_1^{[n]}, \overline e_2^{[n]}, (x\overline e_1 +\overline e_2)^{[n]}\}$ for all $i\in [1,2m-2+k_m]$.

\begin{proof}
Let $S=W\bdot W_1\bdot \ldots\bdot W_{2m-2+k_m}$ be an arbitrary block decomposition. Then Claim A ensures that Conjecture \ref{conj-shortzs} holds for $\varphi(W)$, so there exists a basis $(\overline e_1,\overline e_2)$ for $\varphi(G)$ such that \be\label{dalt}\varphi(W)=\overline e_1^{[n-1]}\bdot \overline e_2^{[n-1]}\bdot (x\overline e_1+\overline e_2)^{[k_n]}\ee for some $x\in [1,n-1]$ with $\gcd(x,n)=1$ and either $x=1$ or $k_n=n-1$ (as we have $k_n\in [2,n-1]$). Since $(\overline e_1,\overline e_2)$ is a basis and $\gcd(x,n)=1$, any $n$-term zero-sum sequence with support contained in $\{\overline e_1,\overline e_2,x\overline e_1+\overline e_2\}$ must have support of size $1$. Consequently, if we can show that $|\supp(\varphi(S))|=3$, then any block decomposition $S=W'\bdot W'_1\bdot\ldots\bdot W'_{2m-2+k_m}$ will have $\varphi(W'_i)\in \{\overline e_1^{[n]}, \overline e_2^{[n]}, (x\overline e_1 +\overline e_2)^{[n]}\}$ for all $i\in [1,2m-2+k_m]$.
Moreover, if $k_n=n-1$, then Conjecture \ref{conj-shortzs} must hold for $\varphi(W')$ with $\supp(\varphi(W'))\subseteq \{\overline e_1,\overline e_2,x\overline e_1+\overline e_2\}$, with each term having multiplicity $n-1$, which forces $\varphi(W')=\varphi(W)=\overline e_1^{[n-1]}\bdot \overline e_2^{[n-1]}\bdot (x\overline e_1+\overline e_2)^{[n-1]}$, while for $k_n\leq n-2$ and $x=1$, Conjecture \ref{conj-shortzs} must hold for $\varphi(W')$ with $\supp(\varphi(W'))\subseteq \{\overline e_1,\overline e_2,\overline e_1+\overline e_2\}$, \  $\overline e_1+(\overline e_1+\overline e_2)\neq \overline e_2$ and $\overline e_2+(\overline e_1+\overline e_2)\neq \overline e_1$ (as $n\geq 3$), ensuring that $\varphi(W')=\varphi(W)=\overline e_1^{[n-1]}\bdot \overline e_2^{[n-1]}\bdot (x\overline e_1+\overline e_2)^{[k_n]}$. In both cases, the claim would be complete.
Thus it suffices to show $|\supp(\varphi(S))|=3$. Assume by contradiction that $|\supp(\varphi(S))|>3$, meaning there is some $g\in \supp(S\bdot W_0^{[-1]})$ with $\varphi(g)\notin \{\overline e_1,\overline e_2,x\overline e_1+\overline e_2\}=\supp(\varphi(W))$, and  w.l.o.g. $g\in \supp(W_1)$.

Suppose there were two distinct elements from $\supp(\varphi(W\bdot W_1))$ each with multiplicity at least $n$ in $\varphi(W\bdot W_1)$.  Then it would be possible to re-factorize $W\bdot W_1=W'\bdot W'_1$ with $|W'_1|=|W_1|=n$, with $\varphi(W'_1)$ a zero-sum sequence having  support of size $1$, and with $W'$ containing a zero-sum subsequence of length $n$ and support $1$. But then $S=W'\bdot W'_1\bdot W_2\bdot\ldots\bdot W_{2m-2+k_m}$ would be a block decomposition with $0\in \Sigma_{\leq n}(\varphi(W'))$, contrary to Claim A. So we conclude there can be at most one term from $\varphi(W\bdot W_1)$ with multiplicity at least $n$.

Suppose $\vp_{\varphi(g)}(\varphi(W\bdot W_1))\geq n-1$. Since $\varphi(g)\notin \supp(\varphi(W))$, this ensures $\varphi(g)$ has multiplicity at least $n-1$ in the $n$-term zero-sum sequence $\varphi(W_1)$, which is only possible if $\varphi(W_1)=\varphi(g)^{[n]}$.
In such case,  all terms in $\varphi(W\bdot W_1\bdot g^{[-1]})$ have multiplicity at most $n-1$.
Since we have $|W\bdot W_1\bdot g^{[-1]}|=3n-3+k_n\geq 3n-1\geq \mathsf s_{\leq n}(\varphi(G))$, it follows that there is a nontrivial subsequence $W'_1\mid W\bdot W_1\bdot g^{[-1]}$ with $\varphi(W'_1)$ a zero-sum sequence of length at most $n$. Setting $W'=W\bdot W_1\bdot (W'_1)^{[-1]}$, it follows that $S=W'\bdot W'_1\bdot W_2\bdot\ldots\bdot W_{2m-2+k_m}$ is a block decomposition with $|\supp(\varphi(W'\bdot W'_1))|=|\supp(\varphi(W\bdot W_1))|=4$ and $|\supp(\varphi(W'_1))|>1$ (as $W'_1\mid W\bdot W_1\bdot g^{[-1]}$ with all terms of $\varphi(W\bdot W_1\bdot g^{[-1]})$ having multiplicity at most $n-1$). Replacing the initial block decomposition by  $S=W'\bdot W'_1\bdot W_2\bdot\ldots\bdot W'_{2m-2+k_m}$ and repeating all arguments from the start (including possibly redefining the elements $\overline e_1$, $\overline e_2$ and $x\overline e_1+\overline e_2$), we see that we can w.l.o.g. assume there is some $g\in \supp(W_1)$ with
$$\varphi(g)\notin \supp(\varphi(W))=\{\overline e_1,\overline e_2,x\overline e_1+\overline e_2\}\quad\und\quad \vp_{\varphi(g)}(\varphi(W\bdot W_1))\leq n-2.$$

Let $g_1,g_2\in \supp(W)$ be elements with $\varphi(g_1)=\overline e_1$ and $\varphi(g_2)=\overline e_2$.
Now we see that  $W\bdot W_1\bdot g_1^{[-1]}\bdot g_2^{[-1]}\bdot g^{[-1]}$ is a sequence of length $3n-5+k_n\geq 3n-3$.
As a result, if $\varphi(W\bdot W_1\bdot g_1^{[-1]}\bdot g_2^{[-1]}\bdot g^{[-1]})$ does not contain a zero-sum subsequence of length at most $n$, then the established case of Conjecture \ref{conj-shortzs}.4 implies that $\supp(\varphi(W\bdot W_1\bdot g_1^{[-1]}\bdot g_2^{[-1]}\bdot g^{[-1]}))=3$ with each of these three terms occurring with multiplicity $n-1$ in  $\varphi(W\bdot W_1\bdot g_1^{[-1]}\bdot g_2^{[-1]}\bdot g^{[-1]})$.
However, since $n\geq 3$, \eqref{dalt} ensures that  $\overline e_1,\,\overline e_2\in \supp(\varphi(W\bdot W_1\bdot g_1^{[-1]}\bdot g_2^{[-1]}\bdot g^{[-1]}))$, meaning $\overline e_1$ and $\overline e_2$ each have multiplicity $n-1$ in $\varphi(W\bdot W_1\bdot g_1^{[-1]}\bdot g_2^{[-1]}\bdot g^{[-1]})$, whence $\vp_{\overline e_1}(\varphi(W\bdot W_1))\geq n$ and $\vp_{\overline e_2}(\varphi(W\bdot W_1))\geq n$ (since $\varphi(g_1)=\overline e_1$ and $\varphi(g_2)=\overline e_2$), contradicting that we showed earlier that at most one term of $\varphi(W\bdot W_1)$ can have multiplicity at least $n$.
Therefore we instead conclude that there is some subsequence $W'_1\mid W\bdot W_1\bdot g_1^{[-1]}\bdot g_2^{[-1]}\bdot g^{[-1]}$ with $\varphi(W'_1)$ a length $n$ zero-sum (in view of Claim A).
Setting $W'=W\bdot W_1\bdot (W'_1)^{[-1]}$, we find that $S=W'\bdot W'_1\bdot W_2\bdot\ldots\bdot W_{2m-2+k_m}$ is a block decomposition with $g,g_1,g_2\in \supp(W')$ and $\varphi(g_1)=\overline e_1$, $\varphi(g_2)=\overline e_2$ and $\varphi(g)\notin \{\overline e_1,\overline e_2,x\overline e_1+\overline e_2\}$.

Applying our initial argument in Claim C using this new block decomposition, we immediately obtain a contradiction unless Conjecture \ref{conj-shortzs} holds for $\varphi(W')$ with $\supp(\varphi(W'))=\{\overline e_1,\overline e_2,\varphi(g)\}$. If $k_n=n-1$, this forces each of the terms $\overline e_1$, $\overline e_2$ and $\varphi(g)$ to have multiplicity $n-1$ in $\varphi(W')$, contradicting that $W'\mid W\bdot W_1$ with the multiplicity of $\varphi(g)$ in $\varphi(W\bdot W_1)$ at most $n-2$ (as shown above). Therefore we must have $k_n\leq n-2$, in which case $x=1$, and then, as $\varphi(g)$ has multiplicity at most $n-2$, the only way Conjecture \ref{conj-shortzs} can hold for $\varphi(W')$ is if it does so with basis $(\overline e_1,\overline e_2)$ and $\varphi(g)=\overline e_1+\overline e_2$, contradicting that $\varphi(g)\notin \{\overline e_1,\overline e_2,x\overline e_1+\overline e_2\}=\{\overline e_1,\overline e_2,\overline e_1+\overline e_2\}$. This completes Claim C.
\end{proof}

We define a term $g\in \supp(S)$ to be \emph{good} if $g,\,h\in \supp(S)$ with $\varphi(g)=\varphi(h)$ implies $g=h$.
 A term $g\in \supp(\varphi(S))$ is \emph{good} if $\supp(S)$ contains exactly one element from $\varphi^{-1}(g)$. Then, for $g\in \supp(S)$, we find that  $\varphi(g)$ is good if and only if $g$ is good.

 Let $(\overline e_1,\overline e_2)$ be an arbitrary basis for $\ker \varphi$ for which Claim C holds with \be\label{supp-jam}\supp(\varphi(S))=\{\overline e_1,\overline e_2,x\overline e_1+\overline e_2\},\ee where $x\in [1,n-1]$ with $\gcd(x,n)=1$ and either $x=1$ or $k_n=n-1$. Let $e_1,\,e_2\in G$ be, for the moment, arbitrary  representatives for $\overline e_1$ and $\overline e_2$, so $\varphi(e_1)=\overline e_1$, $\varphi(e_2)=\overline e_2$ and $\varphi(xe_1+e_2)=x\overline e_1+\overline e_2$.
  We divide the remainder of the proof into two main cases. We remark that the cases $k_m\in [1,m-2]$ could be handled by the methods of either CASE 1 or 2, but there is some simplification to the arguments by including them in CASE 2.

 \subsection*{CASE 1:} $k_m=m-1$.

We begin with the following claim.

\subsection*{Claim D.1} All terms $g\in \supp(S)$ are good.

 \begin{proof}
 Let $S=W\bdot W_1\bdot\ldots\bdot W_{2m-2+k_m}$ be an arbitrary block decomposition. In view of Claim C, let $g_1,g_2,g_3\in \supp(W)$ be arbitrary  with $\varphi(g_1)=\overline e_1$, $\varphi(g_2)=\overline e_2$ and $\varphi(g_3)=x\overline e_1+\overline e_2$.
 Let $S_\sigma=\sigma(W_1)\bdot\ldots\bdot \sigma(W_{2m-2+k_m})$ be the associated sequence, which satisfies  Conjecture \ref{conj-shortzs} by Claim B.
 If $h\in \supp(S\bdot W^{[-1]})$, say $h\in \supp(W_j)$, then $\varphi(h)=\varphi(g_k)$ for some $k\in [1,3]$ by Claim C.
Setting $W'=W\bdot g_k^{[-1]}\bdot h$, $W'_j=W_j\bdot h^{[-1]}\bdot g_k$ and $W'_i=W_i$ for all $i\neq j$, we obtain a new block decomposition $S=W'\bdot W'_1\bdot\ldots\bdot W'_{2m-2+k_m}$ with associated sequence $S'_\sigma=\sigma(W'_1)\bdot\ldots\bdot \sigma(W'_{2m-2+k_m})$ also satisfying Conjecture \ref{conj-shortzs} by Claim B. Since $k_m=m-1\geq 1$, we can then apply Lemma \ref{lem-perturb-solo} to conclude that $S'_\sigma=
S_\sigma$ and $\sigma(W'_j)=\sigma(W_j)$. Since $\sigma(W'_j)=\sigma(W_j)-h+g_k$, this forces $g_k=h$. Ranging over all $h\in \supp(S\bdot W^{[-1]})$ and $g_k\in \supp(W)$ with $\varphi(g_k)=\varphi(h)$ now shows that $\varphi(h)$ is good.

In summary, this argument shows that all terms occurring in $\varphi(S\bdot W^{[-1]})$ are good. Consequently, since Claim C ensures that $|\supp(\varphi(S))|=3$, the only way Claim D can fail is if $|\supp(\varphi(S\bdot W^{[-1]}))|\leq 2$ with all terms in $\supp(\varphi(S\bdot W^{[-1]}))$ good. In view of Claim C, each $\varphi(W_i)$ with $i\in [1,2m-2+k_m]$ consists of a single term from $\{\overline e_1,\overline e_2,x\overline e_1 +\overline e_2\}$ repeated $n$ times.
As a result, if $|\supp(\varphi(S\bdot W^{[-1]}))|\leq 2$, then the Pigeonhole Principle ensures that, among the terms from the $2m-2+k_m=3m-3$  blocks $W_i$ with $i\in [1,2m-2+k_m]$, at least $\lceil\frac{3m-3}{2}\rceil \geq m$ of these blocks $\varphi(W_i)$ must have support equal to the same element, which is then good, ensuring that there is only a single distinct term among these blocks $W_i$. In such case, $S$ has a term with multiplicity at least $mn$, contradicting that $0\notin \Sigma_{\leq 2mn-1-k}(S)$. Therefore, we instead conclude that $|\supp(\varphi(S\bdot W^{[-1]}))|=3$, meaning all $g\in \supp(S)$ are good as explained above.
\end{proof}

Since $k_m=m-1$ and $k=k_mn+k_n\in [2,mn-2]$, we have $k_n\neq n-1$, meaning $k_n\in [2,n-2]$ with $n\geq 4$. This ensures that $x=1$ in Claim C. In view of Claim D.1, \be\label{supp-alpha}\supp(S)=\{e_1,e_2,e_1+e_2+\alpha\}\ee for some  $e_1,e_2\in G$ and $\alpha\in \ker \varphi$ with $\varphi(e_1)=\overline e_1$, $\varphi(e_2)=\overline e_2$ and $\varphi(e_1+e_2+\alpha)=\overline e_1+\overline e_2$.
Let $S=W\bdot W_1\bdot \ldots\bdot W_{2m-2+k_m}$ be an arbitrary block decomposition. Then
Claim C implies that   $$W=e_1^{[n-1]}\bdot  e_2^{[n-1]}\bdot (e_1+e_2+\alpha)^{[k_n]}$$ with $k_n\in [2,n-2]$. Moreover, we can  partition $[1,2m-2+k_m]=I_1\cup I_2\cup I_3$ with $I_j$ consisting of all indices $i\in [1,2m-2+k_m]$ such that $W_i=e_j^{[n]}$ (for $j\in [0,1])$ or such that $W_i=(e_1+e_2+\alpha)^{[n]}$ (for $j=3$). If $I_j=\emptyset$ for some $j\in [1,3]$, then, since $2m-2+k_m\geq 2m-1$ (as $k_m=m-1\geq 1$), the Pigeonhole Principle ensures that $|I_{j'}|\geq m$ for some $j'\in [1,3]\setminus \{j\}$. In such case, $S$ has a term with multiplicity at least $mn$, contradicting that $0\notin \Sigma_{\leq mn}(S)$ by hypothesis. Therefore we may assume each $I_j$ for $j\in [1,3]$ is nonempty.

Since $I_3\neq \emptyset$, let $j\in I_3$. Set $W'=W\bdot e_1^{[-1]}\bdot e_2^{[-1]}\bdot (e_1+e_2+\alpha)$, $W'_j=W_j\bdot (e_1+e_2+\alpha)^{[-1]}\bdot e_1\bdot e_2$, and  $W'_i=W_i$ for all $i\neq j$.
Since $|W'|=|W|-1=2n-3+k_n\geq n-1+2k_n$ (in view of $k_n\in [2,n-2]$), we see that $S=W'\bdot W'_1\bdot\ldots\bdot W'_{2m-2+k_m}$ is a weak block decomposition with associated sequence $S'_\sigma=\sigma(W'_1)\bdot\ldots\bdot \sigma(W'_{2m-2+k_m})$. Since Conjecture \ref{conj-shortzs} is known to always  hold for $k_m=m-1$, it follows in view of Claim B that Conjecture \ref{conj-shortzs} holds for $S'_\sigma$.
As a result, since the sequence $S'_\sigma$ is obtained from $S_\sigma$ by replacing the term $\sigma(W_j)$ by $\sigma(W'_j)=\sigma(W_j)-\alpha$, and since Conjecture \ref{conj-shortzs} holds for $S_\sigma$ by Claim B, we can apply  Lemma \ref{lem-perturb-solo} (as $k_m=m-1\geq 1$) to conclude $\alpha=0$. But now  $\supp(S)\subseteq \{e_1,e_2,e_1+e_2\}$ by \eqref{supp-alpha}, so applying Lemma \ref{lem-genPropB} yields the desired structure for $S$, completing CASE 1.

\subsection*{CASE 2:} $k_m\in [0,m-2]$.

Let $S=W\bdot W_1\bdot \ldots\bdot W_{2m-2+k_m}$ be an arbitrary block decomposition with associated sequence $S_\sigma$.  Then
Claim C implies that   \be\label{summerhug}\varphi(W)=\overline e_1^{[n-1]}\bdot  \overline e_2^{[n-1]}\bdot (x\overline e_1+\overline e_2)^{[k_n]}\ee with $k_n\in [2,n-1]$, with $x\in [1,n-1]$ and $\gcd(x,n)=1$, and with either $x=1$ or $k_n=n-1$. Since $n\geq 3$ (as noted at the start of the proof), $x=n-1$ is only possible if $k_n=n-1$, in which case Claim C also holds replacing the basis $(f_1,f_2)$ by $(f_1,xf_1+f_2)=(f_1,-f_1+f_2)$ with $f_2=f_1+(xf_1+f_2)$. In such case, by using this alternative basis in Claim C, we obtain $x=1<n-1$. This allows us to  w.l.o.g. assume $x\in [1,n-2]$ with $\gcd(x,n)=1$.

If $x=1$,  then \eqref{summerhug} and $k_n\in [2,n-1]$ ensures that there are sequences $U_0\mid W$ and $V_0\mid W$ with $$\varphi(U_0)=\overline e_1^{[n-k_n]}\bdot \overline e_2^{[n-k_n]}\bdot (\overline e_1+\overline e_2)^{[k_n]}\quad\und\quad \varphi(V_0)=\overline e_1^{[n-k_n+1]}\bdot \overline e_2^{[n-k_n+1]}\bdot (\overline e_1+\overline e_2)^{[k_n-1]}.$$ Since $k_n\in [2,n-1]$, we find that
\begin{align}\label{whatsinW0}&\supp(\varphi(U_0))=\supp(\varphi(V_0))=\{\overline e_1, \overline e_2,\overline e_1+\overline e_2\}, \\ \label{whatsinW}&\overline e_1,\overline e_2\in \supp(\varphi(W\bdot U_0^{[-1]}))\quad\und\quad \overline e_1+\overline e_2\in \supp(\varphi(W\bdot V_0^{[-1]})).
\end{align}  On the other hand, if $x\neq 1$, then $x\in [2,n-2]$ with $\gcd(x,n)=1$, $n\geq 5$ and $k_n=n-1$. In such case, let  $x^*\in [2,n-2]$  be the multiplicative inverse of $-x$, so  $$xx^*\equiv -1\mod n.$$
In this case, in view of \eqref{summerhug} and $n\geq 5$, there is a sequence $W_0\mid W$ with $$\varphi(W_0)=\overline e_1\bdot \overline e_2^{[n-x^*]}\bdot (x\overline e_1+\overline e_2)^{[x^*]}.$$
In view of $x^*\in [2,n-2]$ and $k_n=n-1$, we have \begin{align}\label{whatsinW0-x}&\supp(\varphi(W_0))=\supp(\varphi(W\bdot W_0^{[-1]}))=\{\overline e_1, \overline e_2,x\overline e_1+\overline e_2\}\quad \und\\ \label{whatsinW-x}&\vp_{\overline e_1}(\varphi(W\bdot W_0^{[-1]}))=n-2\geq x.
\end{align}
Since $|W_0|=n+1$, $|U_0|=2n-k_n$ and $|V_0|=2n-k_n+1$ are all at most $3n-1-k_n$ (in view of $n\geq 2$ and $k_n\leq n-1\leq 2n-2$), it follows that $S=(W\bdot W_0^{[-1]})\bdot W_0\bdot W_1\ldots\bdot W_{2m-2+k_m}$, $S=(W\bdot U_0^{[-1]})\bdot U_0\bdot W_1\ldots\bdot W_{2m-2+k_m}$ and $S=(W\bdot V_0^{[-1]})\bdot V_0\bdot W_1\ldots\bdot W_{2m-2+k_m}$ are each augmented block decompositions of $S$ (when  they are defined), with respective associated sequences $\sigma(W_0)\bdot S_\sigma$, $\sigma(U_0)\bdot S_\sigma$, and $\sigma(V_0)\bdot S_\sigma$.
In view of Claim B and the case hypothesis $k_m\in [0,m-2]$, Conjecture \ref{conj-shortzs} holds for all of these associated sequences.

We continue with the following claim.

\subsection*{Claim D.2} All terms $g\in \supp(S)$ are good.

 \begin{proof}
 Suppose $x\neq 1$. Consider the augmented block decomposition $S=(W\bdot W_0^{[-1]})\bdot W_0\bdot W_1\ldots\bdot W_{2m-2+k_m}$.  Then \eqref{whatsinW0-x} ensures that there are $g_1,g_2,g_2\in \supp(W\bdot W_0^{[-1]})$ with $\varphi(g_1)=\overline e_1$, $\varphi(g_2)=\overline e_2$ and $\varphi(g_3)=x\overline e_1+\overline e_2$. Taking an arbitrary $g\in \supp(W_0\bdot\ldots \bdot W_{2m-2+k_m})$ and exchanging $g$ with $g_j$, where $g_j$ with $j\in [1,3]$ is the element with $\varphi(g_j)=\varphi(g)$, results in a new augmented block decomposition. Applying Lemma \ref{lem-perturb-solo}, we find that the associated sequence for the new block decomposition must equal that of the original one, forcing $g=g_j$. Ranging over all possible $g_j\in \supp(W\bdot W_0^{[-1]})$ and $g\in \supp(W_0\bdot\ldots\bdot W_{2m-2+k_m})$ with $\varphi(g)=\varphi(g_j)$, it follows that $g$ is good. This shows that all terms from $W_0\bdot\ldots\bdot W_{2m-2+k_m}$ are good, and in view of \eqref{whatsinW0-x}, each possible term $\overline e_1$, $\overline e_2$ and $x\overline e_1+\overline e_2$ occurs in $\varphi(W_0\bdot\ldots\bdot W_{2m-2+k_m})$, ensuring that every $g\in \supp(S)$ is good.

 Next suppose $x=1$. Repeating the above argument using the augmented block decomposition $S=(W\bdot U_0^{[-1]})\bdot U_0\bdot W_1\ldots\bdot W_{2m-2+k_m}$ in place of $S=(W\bdot W_0^{[-1]})\bdot W_0\bdot W_1\ldots\bdot W_{2m-2+k_m}$, and using \eqref{whatsinW} and \eqref{whatsinW0} in place of \eqref{whatsinW0-x} shows that $\overline e_1$ and $\overline e_2$ are both good. Repeating the above argument using the augmented block decomposition $S=(W\bdot V_0^{[-1]})\bdot V_0\bdot W_1\ldots\bdot W_{2m-2+k_m}$ in place of $S=(W\bdot W_0^{[-1]})\bdot W_0\bdot W_1\ldots\bdot W_{2m-2+k_m}$, and using \eqref{whatsinW} and \eqref{whatsinW0} in place of \eqref{whatsinW0-x} shows that $\overline e_1+\overline e_2$ is  good. As $\supp(\varphi(S))=\{\overline e_1,\overline e_2,\overline e_1+\overline e_2\}$, this ensures that all terms $g\in \supp(S)$ are good.
 \end{proof}

In view of Claim D.2, there are $e_1,e_2\in G$ and $\alpha\in \ker\varphi$ such that $$\supp(S)=\{e_1,e_2,xe_1+e_2+\alpha\}$$ with $\varphi(e_1)=\overline e_1$, $\varphi(e_2)=\overline e_2$ and $\varphi(e_1+e_2+\alpha)=x\overline e_1+\overline e_2$.

Suppose $x\neq 1$. Then \eqref{whatsinW0-x} and \eqref{whatsinW-x} ensure that there is  a subsequence $T\mid W\bdot W_0^{[-1]}$  with $T=e_1^{[x]}\bdot e_2$. In view of \eqref{whatsinW0-x}, we have $xe_2+e_3+
\alpha\in \supp(W_0)$, so set $W'_0= W_0\bdot (xe_1+e_2+\alpha)^{[-1]}\bdot T$. Since $|W'_0|=|W_0|+x=n+1+x\leq 2n= 3n-1-k_n$, it follows that $(W\bdot (W'_0)^{[-1]})\bdot W'_0\bdot W_1\bdot\ldots\bdot W_{2m-2+k_m}$ is an augmented block decomposition with  associated sequence satisfying Conjecture \ref{conj-shortzs} by Claim B. Applying Lemma \ref{lem-perturb-solo}, we find that the associated sequence for the new block decomposition must equal that of the original one, which is only possible if $\alpha=0$. We now have $\supp(S)=\{e_1,e_2,xe_1+e_2+\alpha\}=\{e_1,e_2,xe_1+e_2\}$, so applying Lemma \ref{lem-genPropB} yields the desired structure for $S$.

Next Suppose $x=1$. Then \eqref{whatsinW}  ensures that there is  a subsequence $T\mid W\bdot U_0^{[-1]}$  with $T=e_1\bdot e_2$. In view of \eqref{whatsinW0}, we have $e_1+e_2+\alpha\in\supp(U_0)$, so set $U'_0= U_0\bdot (e_1+e_2+\alpha)^{[-1]}\bdot T$. Since $|U'_0|=|U_0|+1=2n-k_n+1\leq 3n-1-k_n$, it follows that $(W\bdot (U'_0)^{[-1]})\bdot U'_0\bdot W_1\bdot\ldots\bdot W_{2m-2+k_m}$ is an augmented block decomposition with  associated sequence satisfying Conjecture \ref{conj-shortzs} by Claim B. Applying Lemma \ref{lem-perturb-solo}, we find that the associated sequence for the new block decomposition must equal that of the original one, which is only possible if $\alpha=0$. As before, we now have $\supp(S)=\{e_1,e_2,e_1+e_2+\alpha\}=\{e_1,e_2,e_1+e_2\}$, so applying Lemma \ref{lem-genPropB} yields the desired structure for $S$, which completes CASE 2 and the proof.
\end{proof}

\begin{proof}[Proof of Theorem \ref{thm-mult}]
Theorem \ref{thm-mult} follows directly from Propositions \ref{prop-main-k_n-01} and \ref{prop-kn-2n-1}
\end{proof}

We conclude the paper by giving the short proofs of Corollaries \ref{cor-mult-2,3}, \ref{cor-mult}  and \ref{cor-bigspec}.

\begin{proof}[Proof of Corollary \ref{cor-mult-2,3}]
It suffices in view of Theorem \ref{thm-mult} to know Conjecture \ref{conj-shortzs} holds for all $k\in [0,p-1]$ in $G=C_p\oplus C_p$, for $p\in \{2,3,5,7\}$. As noted in the introduction,  Conjecture \ref{conj-shortzs} is known for $k\leq 1$,  $k=p-1$ and $k\in [2,\frac{2p+1}{3}]$ in $C_p\oplus C_p$ with $p$ prime. Since,  $p-2\leq \frac{2p+1}{3}$ for $p\leq 7$, this means Conjecture \ref{conj-shortzs} is known for all $k\in [0,p-1]$ in $C_p\oplus C_p$ for $p\leq 7$ prime, as required.
\end{proof}

\begin{proof}[Proof of Corollary \ref{cor-mult}]
In view of Corollary \ref{cor-mult-2,3}, we can assume $n\geq 11$.  As noted in the introduction, Conjecture \ref{conj-shortzs} is known for $k\leq \frac{2n+1}{3}$ in a $p$-group $C_n\oplus C_n$ provided $p\nmid k$ and $n\geq 5$. It remains to show Conjecture \ref{conj-shortzs} holds for $k=rp$ with $r\in [1,\frac{2n+1}{3p}]$. If $n=p$, there is nothing to show, so we assume $n=p^s$ with $s\geq 2$, and proceed by induction on $s$ with the base $s=1$ of the induction complete. We have $rp=k\leq \frac{2p^s+1}{3}$, ensuring that $r\leq \frac{2p^{s-1}+1}{3}$. Thus, by induction hypothesis, Conjecture \ref{conj-shortzs} holds for $r$ in $C_{p^{s-1}}\oplus C_{p^{s-1}}$, while Conjecture \ref{conj-shortzs} holds in general for $k_{p}:=0$ in $C_p\oplus C_p$ (as noted in the introduction). As a result, Theorem \ref{thm-mult} (applied with $n=p$ and $m=p^{s-1}$) implies that Conjecture \ref{conj-shortzs} holds for $k=rp$ in $C_{p^s}\oplus C_{p^s}$, completing the induction and the proof.
\end{proof}

\begin{proof}[Proof of Corollary \ref{cor-bigspec}]
 Write $n=bd$ with $b$ and $d$ proper nontrivial divisors of $n$.
 As noted in the introduction, Conjecture \ref{conj-shortzs} holds for $k=d-1$ and for $k=1$ in $C_d\oplus C_d$;  if $b=2$, then Conjecture \ref{conj-shortzs} holds for $k=b-2=0$ in $C_b\oplus C_b$; and  if $b\geq 3$, then Conjecture \ref{conj-shortzs}  holds for $k+1=b-1$ in $C_b\oplus C_b$. In both of the latter two cases, since $d-1\geq 1$, applying Theorem \ref{thm-mult} using $n=d$ and $m=b$ shows that Conjecture \ref{conj-shortzs} holds for $k=n-d-1=(b-2)d+(d-1)$ in $C_{bd}\oplus C_{bd}=C_n\oplus C_n$, and applying Theorem \ref{thm-mult} using $n=d$ and $m=b$ shows that Conjecture \ref{conj-shortzs} holds for $k=n-2d+1=(b-2)d+1$ in $C_{bd}\oplus C_{bd}=C_n\oplus C_n$, as desired.
\end{proof}

\end{document}